\documentclass{amsart}
\usepackage{amsmath}
\usepackage{amssymb}
\usepackage{amsfonts}

\newtheorem{theorem}{Theorem}
\theoremstyle{plain}

\newtheorem{corollary}{Corollary}

\newtheorem{definition}{Definition}

\newtheorem{lemma}{Lemma}

\newtheorem{remark}{Remark}

\numberwithin{equation}{section}
\input{tcilatex}

\begin{document}
\title[Some Integral Inequalities]{On Some New Hadamard Like Inequalities
for Co-ordinated\ $s-$convex Functions}
\author{M. Emin \"{O}zdemir$^{\blacklozenge }$}
\address{$^{\blacklozenge }$Ataturk University, K.K. Education Faculty,
Department of Mathematics, 25240, Erzurum, Turkey}
\email{emos@atauni.edu.tr}
\author{Mevl\"{u}t Tun\c{c}$^{\triangle }$}
\address{$^{\triangle }$University of Kilis 7 Aral\i k, Faculty of Science
and Arts, Department of Mathematics, 79000, Kilis, Turkey}
\email{mevluttunc@kilis.edu.tr}
\author{Ahmet Ocak Akdemir$^{\spadesuit }$}
\address{$^{\spadesuit }$A\u{g}r\i\ \.{I}brahim \c{C}e\c{c}en University,
Faculty of Science and Arts, Department of Mathematics, 04100, A\u{g}r\i ,
Turkey}
\email{ahmetakdemir@agri.edu.tr}

\begin{abstract}
In this paper, we prove some new inequalities of Hadamard-type for $s-$%
convex functions on the co-ordinates.
\end{abstract}

\maketitle

\section{INTRODUCTION}

Let $f:I\subseteq 
\mathbb{R}
\rightarrow 
\mathbb{R}
$ be a convex function defined on the interval $I$ of real numbers and $a<b.$
The following double inequality;%
\begin{equation*}
f\left( \frac{a+b}{2}\right) \leq \frac{1}{b-a}\dint\limits_{a}^{b}f(x)dx%
\leq \frac{f(a)+f(b)}{2}
\end{equation*}

is well known in the literature as Hadamard's inequality. Both inequalities
hold in the reversed direction if $f$ is concave.

In \cite{alomari}, Alomari and Darus defined s-convex functions on the
co-ordinates as following:

\begin{definition}
Consider the bidimensional interval $\Delta =[a,b]\times \lbrack c,d]$ in $%
\left[ 0,\infty \right) ^{2}$ with $a<b$ and $c<d.$ The mapping $f:\Delta
\rightarrow 
\mathbb{R}
$ is $s-$convex in the second sense on $\Delta $ if%
\begin{equation*}
f(\lambda x+(1-\lambda )z,\lambda y+(1-\lambda )w)\leq \lambda
^{s}f(x,y)+(1-\lambda )^{s}f(z,w)
\end{equation*}%
hold for all $(x,y),(z,w)\in \Delta $ with $\lambda \in \lbrack 0,1].$and
for some fixed $s\in \left( 0,1\right] .$
\end{definition}

\textit{A function }$f:\Delta \rightarrow 
\mathbb{R}
$\textit{\ is }$s-$\textit{convex on }$\Delta $\textit{\ is called
co-ordinated }$s-$\textit{convex on }$\Delta $\textit{\ if the partial
mappings }$f_{y}:[a,b]\rightarrow 
\mathbb{R}
,$\textit{\ }$f_{y}(u)=f(u,y)$\textit{\ and }$f_{x}:[c,d]\rightarrow 
\mathbb{R}
,$\textit{\ }$f_{x}(v)=f(x,v)$\textit{\ are }$s-$\textit{convex for all }$%
y\in \lbrack c,d]$\textit{\ and }$x\in \lbrack a,b]$\textit{\ with some
fixed }$s\in \left( 0,1\right] .$

\bigskip

Recall that the mapping $f:\Delta \rightarrow 
\mathbb{R}
$ is $s-$convex on $\Delta $ if the following inequality holds.

Moreover, in \cite{alomari}, Alomari and Darus established the following
inequalities of Hadamard's type for co-ordinated $s-$convex functions on a
rectangle from the plane $%
\mathbb{R}
^{2}.$

\begin{theorem}
Suppose that $f:\Delta =[a,b]\times \lbrack c,d]\subset \left[ 0,\infty
\right) ^{2}\rightarrow \left[ 0,\infty \right) $ is $s-$convex function on
the co-ordinates on $\Delta $. Then one has the inequalities;%
\begin{eqnarray}
&&4^{s-1}\ f\left( \frac{a+b}{2},\frac{c+d}{2}\right)  \label{1.1} \\
&\leq &2^{s-2}\left[ \frac{1}{b-a}\int_{a}^{b}f(x,\frac{c+d}{2}dx+\frac{1}{%
d-c}\int_{c}^{d}f\left( \frac{a+b}{2},y\right) dy\right]  \notag \\
&\leq &\frac{1}{\left( b-a\right) \left( d-c\right) }\int_{a}^{b}%
\int_{c}^{d}f\left( x,y\right) dydx  \notag \\
&\leq &\frac{1}{2\left( s+1\right) }\left[ \frac{1}{b-a}\int_{a}^{b}f\left(
x,c\right) dx+\frac{1}{b-a}\int_{a}^{b}f\left( x,d\right) dx\right.  \notag
\\
&&\left. +\frac{1}{d-c}\int_{c}^{d}f\left( a,y\right) dy+\frac{1}{d-c}%
\int_{c}^{d}f\left( b,y\right) dy\right]  \notag \\
&\leq &\frac{f\left( a,c\right) +f\left( a,d\right) +f\left( b,c\right)
+f\left( b,d\right) }{\left( s+1\right) ^{2}}.  \notag
\end{eqnarray}
\end{theorem}

Similar results can be found in (\cite{SS}-\cite{ozdmr}).

\bigskip

However, \cite{ozdmr} \"{O}zdemir et.al. established the following lemma for
twice partial differentiable mapping on $\Delta =\left[ a,b\right] \times %
\left[ c,d\right] .$

\begin{lemma}
\label{l1}Let $f:\Delta =\left[ a,b\right] \times \left[ c,d\right]
\rightarrow 
\mathbb{R}
$ be a twice partial differentiable mapping on $\Delta =\left[ a,b\right]
\times \left[ c,d\right] .$ If $\frac{\partial ^{2}f}{\partial t\partial
\lambda }\in L\left( \Delta \right) ,$ then the following equality holds:%
\begin{eqnarray*}
&&\frac{1}{\left( b-a\right) \left( d-c\right) }\left[ A-\left( x-a\right)
\int\limits_{c}^{d}f\left( a,v\right) dv-\left( b-x\right)
\int\limits_{c}^{d}f\left( b,v\right) dv\right. \\
&&-\left( d-y\right) \int\limits_{a}^{b}f\left( u,d\right) du-\left(
y-c\right) \int\limits_{a}^{b}f\left( u,c\right)
du+\int\limits_{a}^{b}\int\limits_{c}^{d}f\left( u,v\right) dudv \\
&=&\frac{\left( x-a\right) ^{2}\left( y-c\right) ^{2}}{\left( b-a\right)
\left( d-c\right) }\int\limits_{0}^{1}\int\limits_{0}^{1}\left( t-1\right)
\left( \lambda -1\right) \frac{\partial ^{2}f}{\partial t\partial \lambda }%
\left( tx+\left( 1-t\right) a,\lambda y+\left( 1-\lambda \right) c\right)
d\lambda dt \\
&&+\frac{\left( x-a\right) ^{2}\left( d-y\right) ^{2}}{\left( b-a\right)
\left( d-c\right) }\int\limits_{0}^{1}\int\limits_{0}^{1}\left( t-1\right)
\left( 1-\lambda \right) \frac{\partial ^{2}f}{\partial t\partial \lambda }%
\left( tx+\left( 1-t\right) a,\lambda y+\left( 1-\lambda \right) d\right)
d\lambda dt \\
&&+\frac{\left( b-x\right) ^{2}\left( y-c\right) ^{2}}{\left( b-a\right)
\left( d-c\right) }\int\limits_{0}^{1}\int\limits_{0}^{1}\left( 1-t\right)
\left( \lambda -1\right) \frac{\partial ^{2}f}{\partial t\partial \lambda }%
\left( tx+\left( 1-t\right) b,\lambda y+\left( 1-\lambda \right) c\right)
d\lambda dt \\
&&+\frac{\left( b-x\right) ^{2}\left( d-y\right) ^{2}}{\left( b-a\right)
\left( d-c\right) }\int\limits_{0}^{1}\int\limits_{0}^{1}\left( 1-t\right)
\left( 1-\lambda \right) \frac{\partial ^{2}f}{\partial t\partial \lambda }%
\left( tx+\left( 1-t\right) b,\lambda y+\left( 1-\lambda \right) d\right)
d\lambda dt
\end{eqnarray*}%
where 
\begin{eqnarray*}
A &=&\frac{\left( x-a\right) \left( y-c\right) f\left( a,c\right) +\left(
x-a\right) \left( d-y\right) f\left( a,d\right) }{\left( b-a\right) \left(
d-c\right) } \\
&&+\frac{\left( b-x\right) \left( y-c\right) f\left( b,c\right) +\left(
b-x\right) \left( d-y\right) f\left( b,d\right) }{\left( b-a\right) \left(
d-c\right) }.
\end{eqnarray*}
\end{lemma}

The main purpose of this paper is to prove some new inequalities of
Hadamard-type for $s-$convex functions on the co-ordinates.

\section{MAIN RESULTS}

\begin{theorem}
\label{t1}Let $f:\Delta =\left[ a,b\right] \times \left[ c,d\right]
\rightarrow 
\mathbb{R}
$ be a partial differentiable mapping on $\Delta =\left[ a,b\right] \times %
\left[ c,d\right] $ and $\frac{\partial ^{2}f}{\partial t\partial \lambda }%
\in L\left( \Delta \right) $. If $\left\vert \frac{\partial ^{2}f}{\partial
t\partial \lambda }\right\vert $ is a $s-$convex function on the
co-ordinates on $\Delta ,$ for some fixed $s\in \left( 0,1\right] $, then
the following inequality holds;%
\begin{eqnarray}
&&  \label{21} \\
&&\left\vert \frac{1}{\left( b-a\right) \left( d-c\right) }\left[ A-\left(
x-a\right) \int\limits_{c}^{d}f\left( a,v\right) dv-\left( b-x\right)
\int\limits_{c}^{d}f\left( b,v\right) dv\right. \right.  \notag \\
&&\left. \left. -\left( d-y\right) \int\limits_{a}^{b}f\left( u,d\right)
du-\left( y-c\right) \int\limits_{a}^{b}f\left( u,c\right)
du+\int\limits_{a}^{b}\int\limits_{c}^{d}f\left( u,v\right) dudv\right]
\right\vert  \notag \\
&\leq &\frac{1}{\left( b-a\right) \left( d-c\right) \left( s+2\right) ^{2}} 
\notag \\
&&\left[ \left( \frac{\left( \left( x-a\right) ^{2}+\left( b-x\right)
^{2}\right) \left( \left( y-c\right) ^{2}+\left( d-y\right) ^{2}\right) }{%
\left( s+1\right) ^{2}}\right) \left\vert \frac{\partial ^{2}f}{\partial
t\partial \lambda }\left( x,y\right) \right\vert \right.  \notag \\
&&+\left( \frac{\left( x-a\right) ^{2}\left( \left( y-c\right) ^{2}+\left(
d-y\right) ^{2}\right) }{\left( s+1\right) }\right) \left\vert \frac{%
\partial ^{2}f}{\partial t\partial \lambda }\left( a,y\right) \right\vert 
\notag \\
&&+\left( \frac{\left( b-x\right) ^{2}\left( \left( y-c\right) ^{2}+\left(
d-y\right) ^{2}\right) }{\left( s+1\right) }\right) \left\vert \frac{%
\partial ^{2}f}{\partial t\partial \lambda }\left( b,y\right) \right\vert 
\notag \\
&&+\left( \frac{\left( y-c\right) ^{2}\left( \left( x-a\right) ^{2}+\left(
b-x\right) ^{2}\right) }{\left( s+1\right) }\right) \left\vert \frac{%
\partial ^{2}f}{\partial t\partial \lambda }\left( x,c\right) \right\vert 
\notag \\
&&+\left( \frac{\left( d-y\right) ^{2}\left( \left( x-a\right) ^{2}+\left(
b-x\right) ^{2}\right) }{\left( s+1\right) }\right) \left\vert \frac{%
\partial ^{2}f}{\partial t\partial \lambda }\left( x,d\right) \right\vert 
\notag \\
&&+\left( x-a\right) ^{2}\left( y-c\right) ^{2}\left\vert \frac{\partial
^{2}f}{\partial t\partial \lambda }\left( a,c\right) \right\vert +\left(
x-a\right) ^{2}\left( d-y\right) ^{2}\left\vert \frac{\partial ^{2}f}{%
\partial t\partial \lambda }\left( a,d\right) \right\vert  \notag \\
&&\left. +\left( b-x\right) ^{2}\left( y-c\right) ^{2}\left\vert \frac{%
\partial ^{2}f}{\partial t\partial \lambda }\left( b,c\right) \right\vert
+\left( b-x\right) ^{2}\left( d-y\right) ^{2}\left\vert \frac{\partial ^{2}f%
}{\partial t\partial \lambda }\left( b,d\right) \right\vert \right]  \notag
\end{eqnarray}%
where A is as above.
\end{theorem}

\begin{proof}
From Lemma \ref{l1} and using the property of modulus, we have%
\begin{eqnarray*}
&&\left\vert \frac{1}{\left( b-a\right) \left( d-c\right) }\left[ A-\left(
x-a\right) \int\limits_{c}^{d}f\left( a,v\right) dv-\left( b-x\right)
\int\limits_{c}^{d}f\left( b,v\right) dv\right. \right.  \\
&&\left. \left. -\left( d-y\right) \int\limits_{a}^{b}f\left( u,d\right)
du-\left( y-c\right) \int\limits_{a}^{b}f\left( u,c\right)
du+\int\limits_{a}^{b}\int\limits_{c}^{d}f\left( u,v\right) dudv\right]
\right\vert  \\
&\leq &\frac{\left( x-a\right) ^{2}\left( y-c\right) ^{2}}{\left( b-a\right)
\left( d-c\right) } \\
&&\times \int\limits_{0}^{1}\int\limits_{0}^{1}\left\vert \left( t-1\right)
\left( \lambda -1\right) \right\vert \left\vert \frac{\partial ^{2}f}{%
\partial t\partial \lambda }\left( tx+\left( 1-t\right) a,\lambda y+\left(
1-\lambda \right) c\right) \right\vert d\lambda dt \\
&&+\frac{\left( x-a\right) ^{2}\left( d-y\right) ^{2}}{\left( b-a\right)
\left( d-c\right) } \\
&&\times \int\limits_{0}^{1}\int\limits_{0}^{1}\left\vert \left( t-1\right)
\left( 1-\lambda \right) \right\vert \left\vert \frac{\partial ^{2}f}{%
\partial t\partial \lambda }\left( tx+\left( 1-t\right) a,\lambda y+\left(
1-\lambda \right) d\right) \right\vert d\lambda dt \\
&&+\frac{\left( b-x\right) ^{2}\left( y-c\right) ^{2}}{\left( b-a\right)
\left( d-c\right) } \\
&&\times \int\limits_{0}^{1}\int\limits_{0}^{1}\left\vert \left( 1-t\right)
\left( \lambda -1\right) \right\vert \left\vert \frac{\partial ^{2}f}{%
\partial t\partial \lambda }\left( tx+\left( 1-t\right) b,\lambda y+\left(
1-\lambda \right) c\right) \right\vert d\lambda dt \\
&&+\frac{\left( b-x\right) ^{2}\left( d-y\right) ^{2}}{\left( b-a\right)
\left( d-c\right) } \\
&&\times \int\limits_{0}^{1}\int\limits_{0}^{1}\left\vert \left( 1-t\right)
\left( 1-\lambda \right) \right\vert \left\vert \frac{\partial ^{2}f}{%
\partial t\partial \lambda }\left( tx+\left( 1-t\right) b,\lambda y+\left(
1-\lambda \right) d\right) \right\vert d\lambda dt.
\end{eqnarray*}%
Since $\left\vert \frac{\partial ^{2}f}{\partial t\partial s}\right\vert $
is co-ordinated $s-$convex, for some fixed $s\in \left( 0,1\right] $, we can
write%
\begin{eqnarray*}
&&\left\vert \frac{1}{\left( b-a\right) \left( d-c\right) }\left[ A-\left(
x-a\right) \int\limits_{c}^{d}f\left( a,v\right) dv-\left( b-x\right)
\int\limits_{c}^{d}f\left( b,v\right) dv\right. \right.  \\
&&\left. \left. -\left( d-y\right) \int\limits_{a}^{b}f\left( u,d\right)
du-\left( y-c\right) \int\limits_{a}^{b}f\left( u,c\right)
du+\int\limits_{a}^{b}\int\limits_{c}^{d}f\left( u,v\right) dudv\right]
\right\vert  \\
&\leq &\frac{\left( x-a\right) ^{2}\left( y-c\right) ^{2}}{\left( b-a\right)
\left( d-c\right) }\int\limits_{0}^{1}\left\vert \left( \lambda -1\right)
\right\vert  \\
&&\times \left[ \int\limits_{0}^{1}\left( t-1\right) t^{s}\left\vert \frac{%
\partial ^{2}f}{\partial t\partial \lambda }\left( x,\lambda y+\left(
1-\lambda \right) c\right) \right\vert dt+\int\limits_{0}^{1}\left(
t-1\right) \left( 1-t\right) ^{s}\left\vert \frac{\partial ^{2}f}{\partial
t\partial \lambda }\left( a,\lambda y+\left( 1-\lambda \right) c\right)
\right\vert dt\right] d\lambda 
\end{eqnarray*}%
\begin{eqnarray*}
&& \\
&&+\frac{\left( x-a\right) ^{2}\left( d-y\right) ^{2}}{\left( b-a\right)
\left( d-c\right) }\int\limits_{0}^{1}\left\vert \left( 1-\lambda \right)
\right\vert  \\
&&\times \left[ \int\limits_{0}^{1}\left( t-1\right) t^{s}\left\vert \frac{%
\partial ^{2}f}{\partial t\partial \lambda }\left( x,\lambda y+\left(
1-\lambda \right) d\right) \right\vert dt+\int\limits_{0}^{1}\left(
t-1\right) \left( 1-t\right) ^{s}\left\vert \frac{\partial ^{2}f}{\partial
t\partial \lambda }\left( a,\lambda y+\left( 1-\lambda \right) d\right)
\right\vert dt\right] d\lambda  \\
&&+\frac{\left( b-x\right) ^{2}\left( y-c\right) ^{2}}{\left( b-a\right)
\left( d-c\right) }\int\limits_{0}^{1}\left\vert \left( \lambda -1\right)
\right\vert  \\
&&\times \left[ \int\limits_{0}^{1}\left( 1-t\right) t^{s}\left\vert \frac{%
\partial ^{2}f}{\partial t\partial \lambda }\left( x,\lambda y+\left(
1-\lambda \right) c\right) \right\vert dt+\int\limits_{0}^{1}\left(
1-t\right) \left( 1-t\right) ^{s}\left\vert \frac{\partial ^{2}f}{\partial
t\partial \lambda }\left( b,\lambda y+\left( 1-\lambda \right) c\right)
\right\vert dt\right] d\lambda  \\
&&+\frac{\left( b-x\right) ^{2}\left( d-y\right) ^{2}}{\left( b-a\right)
\left( d-c\right) }\int\limits_{0}^{1}\left\vert \left( 1-\lambda \right)
\right\vert  \\
&&\times \left[ \int\limits_{0}^{1}\left( 1-t\right) t^{s}\left\vert \frac{%
\partial ^{2}f}{\partial t\partial \lambda }\left( x,\lambda y+\left(
1-\lambda \right) d\right) \right\vert dt+\int\limits_{0}^{1}\left(
1-t\right) \left( 1-t\right) ^{s}\left\vert \frac{\partial ^{2}f}{\partial
t\partial \lambda }\left( b,\lambda y+\left( 1-\lambda \right) d\right)
\right\vert dt\right] d\lambda .
\end{eqnarray*}%
By computing these integrals, we obtain%
\begin{eqnarray*}
&&\left\vert \frac{1}{\left( b-a\right) \left( d-c\right) }\left[ A-\left(
x-a\right) \int\limits_{c}^{d}f\left( a,v\right) dv-\left( b-x\right)
\int\limits_{c}^{d}f\left( b,v\right) dv\right. \right.  \\
&&\left. \left. -\left( d-y\right) \int\limits_{a}^{b}f\left( u,d\right)
du-\left( y-c\right) \int\limits_{a}^{b}f\left( u,c\right)
du+\int\limits_{a}^{b}\int\limits_{c}^{d}f\left( u,v\right) dudv\right]
\right\vert  \\
&\leq &\frac{\left( x-a\right) ^{2}\left( y-c\right) ^{2}}{\left( b-a\right)
\left( d-c\right) }\int\limits_{0}^{1}\left\vert \lambda -1\right\vert \left[
\frac{-1}{\left( s+1\right) \left( s+2\right) }\left\vert \frac{\partial
^{2}f}{\partial t\partial \lambda }\left( x,\lambda y+\left( 1-\lambda
\right) c\right) \right\vert \right.  \\
&&\left. -\frac{1}{s+2}\left\vert \frac{\partial ^{2}f}{\partial t\partial
\lambda }\left( a,\lambda y+\left( 1-\lambda \right) c\right) \right\vert %
\right] d\lambda  \\
&&+\frac{\left( x-a\right) ^{2}\left( d-y\right) ^{2}}{\left( b-a\right)
\left( d-c\right) }\int\limits_{0}^{1}\left\vert 1-\lambda \right\vert \left[
\frac{-1}{\left( s+1\right) \left( s+2\right) }\left\vert \frac{\partial
^{2}f}{\partial t\partial \lambda }\left( x,\lambda y+\left( 1-\lambda
\right) d\right) \right\vert \right.  \\
&&\left. -\frac{1}{s+2}\left\vert \frac{\partial ^{2}f}{\partial t\partial
\lambda }\left( a,\lambda y+\left( 1-\lambda \right) d\right) \right\vert %
\right] d\lambda  \\
&&+\frac{\left( b-x\right) ^{2}\left( y-c\right) ^{2}}{\left( b-a\right)
\left( d-c\right) }\int\limits_{0}^{1}\left\vert \lambda -1\right\vert \left[
\frac{-1}{\left( s+1\right) \left( s+2\right) }\left\vert \frac{\partial
^{2}f}{\partial t\partial \lambda }\left( x,\lambda y+\left( 1-\lambda
\right) c\right) \right\vert \right.  \\
&&\left. -\frac{1}{s+2}\left\vert \frac{\partial ^{2}f}{\partial t\partial
\lambda }\left( b,\lambda y+\left( 1-\lambda \right) c\right) \right\vert %
\right] d\lambda  \\
&&
\end{eqnarray*}%
\begin{eqnarray*}
&&+\frac{\left( b-x\right) ^{2}\left( d-y\right) ^{2}}{\left( b-a\right)
\left( d-c\right) }\int\limits_{0}^{1}\left\vert 1-\lambda \right\vert \left[
\frac{-1}{\left( s+1\right) \left( s+2\right) }\left\vert \frac{\partial
^{2}f}{\partial t\partial \lambda }\left( x,\lambda y+\left( 1-\lambda
\right) d\right) \right\vert \right.  \\
&&\left. -\frac{1}{s+2}\left\vert \frac{\partial ^{2}f}{\partial t\partial
\lambda }\left( b,\lambda y+\left( 1-\lambda \right) d\right) \right\vert %
\right] d\lambda .
\end{eqnarray*}%
Using co-ordinated $s-$convexity of $\left\vert \frac{\partial ^{2}f}{%
\partial t\partial \lambda }\right\vert $ again and computing all integrals,
we obtain%
\begin{eqnarray*}
&&\left\vert \frac{1}{\left( b-a\right) \left( d-c\right) }\left[ A-\left(
x-a\right) \int\limits_{c}^{d}f\left( a,v\right) dv-\left( b-x\right)
\int\limits_{c}^{d}f\left( b,v\right) dv\right. \right.  \\
&&\left. \left. -\left( d-y\right) \int\limits_{a}^{b}f\left( u,d\right)
du-\left( y-c\right) \int\limits_{a}^{b}f\left( u,c\right)
du+\int\limits_{a}^{b}\int\limits_{c}^{d}f\left( u,v\right) dudv\right]
\right\vert  \\
&\leq &\frac{\left( x-a\right) ^{2}\left( y-c\right) ^{2}}{\left( b-a\right)
\left( d-c\right) } \\
&&\times \left\{ \int\limits_{0}^{1}\left\vert \lambda -1\right\vert \left[ 
\frac{-1}{\left( s+1\right) \left( s+2\right) }\left( \lambda ^{s}\left\vert 
\frac{\partial ^{2}f}{\partial t\partial \lambda }\left( x,y\right)
\right\vert +\left( 1-\lambda \right) ^{s}\left\vert \frac{\partial ^{2}f}{%
\partial t\partial \lambda }\left( x,c\right) \right\vert \right) \right]
d\lambda \right.  \\
&&+\left. \int\limits_{0}^{1}\left\vert \lambda -1\right\vert \left[ -\frac{1%
}{s+2}\left( \lambda ^{s}\left\vert \frac{\partial ^{2}f}{\partial t\partial
\lambda }\left( a,y\right) \right\vert +\left( 1-\lambda \right)
^{s}\left\vert \frac{\partial ^{2}f}{\partial t\partial \lambda }\left(
a,c\right) \right\vert \right) \right] d\lambda \right\}  \\
&&+\frac{\left( x-a\right) ^{2}\left( d-y\right) ^{2}}{\left( b-a\right)
\left( d-c\right) } \\
&&\times \left\{ \int\limits_{0}^{1}\left\vert 1-\lambda \right\vert \left[ 
\frac{-1}{\left( s+1\right) \left( s+2\right) }\left( \lambda ^{s}\left\vert 
\frac{\partial ^{2}f}{\partial t\partial \lambda }\left( x,y\right)
\right\vert +\left( 1-\lambda \right) ^{s}\left\vert \frac{\partial ^{2}f}{%
\partial t\partial \lambda }\left( x,d\right) \right\vert \right) \right]
d\lambda \right.  \\
&&+\left. \int\limits_{0}^{1}\left\vert 1-\lambda \right\vert \left[ -\frac{1%
}{s+2}\left( \lambda ^{s}\left\vert \frac{\partial ^{2}f}{\partial t\partial
\lambda }\left( a,y\right) \right\vert +\left( 1-\lambda \right)
^{s}\left\vert \frac{\partial ^{2}f}{\partial t\partial \lambda }\left(
a,d\right) \right\vert \right) \right] d\lambda \right\}  \\
&&+\frac{\left( b-x\right) ^{2}\left( y-c\right) ^{2}}{\left( b-a\right)
\left( d-c\right) } \\
&&\times \left\{ \int\limits_{0}^{1}\left\vert \lambda -1\right\vert \left[ 
\frac{-1}{\left( s+1\right) \left( s+2\right) }\left( \lambda ^{s}\left\vert 
\frac{\partial ^{2}f}{\partial t\partial \lambda }\left( x,y\right)
\right\vert +\left( 1-\lambda \right) ^{s}\left\vert \frac{\partial ^{2}f}{%
\partial t\partial \lambda }\left( x,c\right) \right\vert \right) \right]
d\lambda \right.  \\
&&+\left. \int\limits_{0}^{1}\left\vert \lambda -1\right\vert \left[ -\frac{1%
}{s+2}\left( \lambda ^{s}\left\vert \frac{\partial ^{2}f}{\partial t\partial
\lambda }\left( b,y\right) \right\vert +\left( 1-\lambda \right)
^{s}\left\vert \frac{\partial ^{2}f}{\partial t\partial \lambda }\left(
b,c\right) \right\vert \right) \right] d\lambda \right\} 
\end{eqnarray*}%
\begin{eqnarray*}
&&+\frac{\left( b-x\right) ^{2}\left( d-y\right) ^{2}}{\left( b-a\right)
\left( d-c\right) } \\
&&\times \left\{ \int\limits_{0}^{1}\left\vert 1-\lambda \right\vert \left[ 
\frac{-1}{\left( s+1\right) \left( s+2\right) }\left( \lambda ^{s}\left\vert 
\frac{\partial ^{2}f}{\partial t\partial \lambda }\left( x,y\right)
\right\vert +\left( 1-\lambda \right) ^{s}\left\vert \frac{\partial ^{2}f}{%
\partial t\partial \lambda }\left( x,d\right) \right\vert \right) \right]
d\lambda \right.  \\
&&+\left. \int\limits_{0}^{1}\left\vert 1-\lambda \right\vert \left[ -\frac{1%
}{s+2}\left( \lambda ^{s}\left\vert \frac{\partial ^{2}f}{\partial t\partial
\lambda }\left( b,y\right) \right\vert +\left( 1-\lambda \right)
^{s}\left\vert \frac{\partial ^{2}f}{\partial t\partial \lambda }\left(
b,d\right) \right\vert \right) \right] \right\} 
\end{eqnarray*}%
then, we get%
\begin{eqnarray*}
&&\left\vert \frac{1}{\left( b-a\right) \left( d-c\right) }\left[ A-\left(
x-a\right) \int\limits_{c}^{d}f\left( a,v\right) dv-\left( b-x\right)
\int\limits_{c}^{d}f\left( b,v\right) dv\right. \right.  \\
&&\left. \left. -\left( d-y\right) \int\limits_{a}^{b}f\left( u,d\right)
du-\left( y-c\right) \int\limits_{a}^{b}f\left( u,c\right)
du+\int\limits_{a}^{b}\int\limits_{c}^{d}f\left( u,v\right) dudv\right]
\right\vert  \\
&\leq &\frac{\left( x-a\right) ^{2}\left( y-c\right) ^{2}}{\left( b-a\right)
\left( d-c\right) } \\
&&\times \left[ \frac{1}{\left( s+1\right) ^{2}\left( s+2\right) ^{2}}%
\left\vert \frac{\partial ^{2}f}{\partial t\partial \lambda }\left(
x,y\right) \right\vert +\frac{1}{\left( s+1\right) \left( s+2\right) ^{2}}%
\left\vert \frac{\partial ^{2}f}{\partial t\partial \lambda }\left(
x,c\right) \right\vert \right.  \\
&&\left. +\frac{1}{\left( s+1\right) \left( s+2\right) ^{2}}\left\vert \frac{%
\partial ^{2}f}{\partial t\partial \lambda }\left( a,y\right) \right\vert +%
\frac{1}{\left( s+2\right) ^{2}}\left\vert \frac{\partial ^{2}f}{\partial
t\partial \lambda }\left( a,c\right) \right\vert \right\}  \\
&&+\frac{\left( x-a\right) ^{2}\left( d-y\right) ^{2}}{\left( b-a\right)
\left( d-c\right) } \\
&&\times \left\{ \frac{1}{\left( s+1\right) ^{2}\left( s+2\right) ^{2}}%
\left\vert \frac{\partial ^{2}f}{\partial t\partial \lambda }\left(
x,y\right) \right\vert +\frac{1}{\left( s+1\right) \left( s+2\right) ^{2}}%
\left\vert \frac{\partial ^{2}f}{\partial t\partial \lambda }\left(
x,d\right) \right\vert \right.  \\
&&\left. +\frac{1}{\left( s+1\right) \left( s+2\right) ^{2}}\left\vert \frac{%
\partial ^{2}f}{\partial t\partial \lambda }\left( a,y\right) \right\vert +%
\frac{1}{\left( s+2\right) ^{2}}\left\vert \frac{\partial ^{2}f}{\partial
t\partial \lambda }\left( a,d\right) \right\vert \right\}  \\
&&+\frac{\left( b-x\right) ^{2}\left( y-c\right) ^{2}}{\left( b-a\right)
\left( d-c\right) } \\
&&\times \left\{ \frac{1}{\left( s+1\right) ^{2}\left( s+2\right) ^{2}}%
\left\vert \frac{\partial ^{2}f}{\partial t\partial \lambda }\left(
x,y\right) \right\vert +\frac{1}{\left( s+1\right) \left( s+2\right) ^{2}}%
\left\vert \frac{\partial ^{2}f}{\partial t\partial \lambda }\left(
x,c\right) \right\vert \right.  \\
&&\left. +\frac{1}{\left( s+1\right) \left( s+2\right) ^{2}}\left\vert \frac{%
\partial ^{2}f}{\partial t\partial \lambda }\left( b,y\right) \right\vert +%
\frac{1}{\left( s+2\right) ^{2}}\left\vert \frac{\partial ^{2}f}{\partial
t\partial \lambda }\left( b,c\right) \right\vert \right\} 
\end{eqnarray*}%
\begin{eqnarray*}
&&+\frac{\left( b-x\right) ^{2}\left( d-y\right) ^{2}}{\left( b-a\right)
\left( d-c\right) } \\
&&\times \left\{ \frac{1}{\left( s+1\right) ^{2}\left( s+2\right) ^{2}}%
\left\vert \frac{\partial ^{2}f}{\partial t\partial \lambda }\left(
x,y\right) \right\vert +\frac{1}{\left( s+1\right) \left( s+2\right) ^{2}}%
\left\vert \frac{\partial ^{2}f}{\partial t\partial \lambda }\left(
x,d\right) \right\vert \right.  \\
&&\left. +\frac{1}{\left( s+1\right) \left( s+2\right) ^{2}}\left\vert \frac{%
\partial ^{2}f}{\partial t\partial \lambda }\left( b,y\right) \right\vert +%
\frac{1}{\left( s+2\right) ^{2}}\left\vert \frac{\partial ^{2}f}{\partial
t\partial \lambda }\left( b,d\right) \right\vert \right\} .
\end{eqnarray*}%
so,%
\begin{eqnarray*}
&&\left\vert \frac{1}{\left( b-a\right) \left( d-c\right) }\left[ A-\left(
x-a\right) \int\limits_{c}^{d}f\left( a,v\right) dv-\left( b-x\right)
\int\limits_{c}^{d}f\left( b,v\right) dv\right. \right.  \\
&&\left. \left. -\left( d-y\right) \int\limits_{a}^{b}f\left( u,d\right)
du-\left( y-c\right) \int\limits_{a}^{b}f\left( u,c\right)
du+\int\limits_{a}^{b}\int\limits_{c}^{d}f\left( u,v\right) dudv\right]
\right\vert  \\
&\leq &\left[ \left( \frac{\left( \left( x-a\right) ^{2}+\left( b-x\right)
^{2}\right) \left( \left( y-c\right) ^{2}+\left( d-y\right) ^{2}\right) }{%
\left( b-a\right) \left( d-c\right) \left( s+1\right) ^{2}\left( s+2\right)
^{2}}\right) \left\vert \frac{\partial ^{2}f}{\partial t\partial s}\left(
x,y\right) \right\vert \right.  \\
&&+\left( \frac{\left( x-a\right) ^{2}\left( \left( y-c\right) ^{2}+\left(
d-y\right) ^{2}\right) }{\left( b-a\right) \left( d-c\right) \left(
s+1\right) \left( s+2\right) ^{2}}\right) \left\vert \frac{\partial ^{2}f}{%
\partial t\partial s}\left( a,y\right) \right\vert  \\
&&+\left( \frac{\left( b-x\right) ^{2}\left( \left( y-c\right) ^{2}+\left(
d-y\right) ^{2}\right) }{\left( b-a\right) \left( d-c\right) \left(
s+1\right) \left( s+2\right) ^{2}}\right) \left\vert \frac{\partial ^{2}f}{%
\partial t\partial s}\left( b,y\right) \right\vert  \\
&&+\left( \frac{\left( y-c\right) ^{2}\left( \left( x-a\right) ^{2}+\left(
b-x\right) ^{2}\right) }{\left( b-a\right) \left( d-c\right) \left(
s+1\right) \left( s+2\right) ^{2}}\right) \left\vert \frac{\partial ^{2}f}{%
\partial t\partial s}\left( x,c\right) \right\vert  \\
&&+\left( \frac{\left( d-y\right) ^{2}\left( \left( x-a\right) ^{2}+\left(
b-x\right) ^{2}\right) }{\left( b-a\right) \left( d-c\right) \left(
s+1\right) \left( s+2\right) ^{2}}\right) \left\vert \frac{\partial ^{2}f}{%
\partial t\partial s}\left( x,d\right) \right\vert  \\
&&+\frac{\left( x-a\right) ^{2}\left( y-c\right) ^{2}}{\left( b-a\right)
\left( d-c\right) \left( s+2\right) ^{2}}\left\vert \frac{\partial ^{2}f}{%
\partial t\partial s}\left( a,c\right) \right\vert +\frac{\left( x-a\right)
^{2}\left( d-y\right) ^{2}}{\left( b-a\right) \left( d-c\right) \left(
s+2\right) ^{2}}\left\vert \frac{\partial ^{2}f}{\partial t\partial s}\left(
a,d\right) \right\vert  \\
&&\left. +\frac{\left( b-x\right) ^{2}\left( y-c\right) ^{2}}{\left(
b-a\right) \left( d-c\right) \left( s+2\right) ^{2}}\left\vert \frac{%
\partial ^{2}f}{\partial t\partial s}\left( b,c\right) \right\vert +\frac{%
\left( b-x\right) ^{2}\left( d-y\right) ^{2}}{\left( b-a\right) \left(
d-c\right) \left( s+2\right) ^{2}}\left\vert \frac{\partial ^{2}f}{\partial
t\partial s}\left( b,d\right) \right\vert \right] .
\end{eqnarray*}%
Which completes the proof.
\end{proof}

\begin{corollary}
\label{c1}(1) Under the assumptions of Theorem \ref{t1}, if we choose $x=a,$ 
$y=c,$ we obtain the following inequality;%
\begin{eqnarray*}
&&\frac{1}{\left( b-a\right) \left( d-c\right) }\left\vert f\left(
b,d\right) -\left( b-a\right) \int\limits_{c}^{d}f\left( b,v\right)
dv-\left( d-c\right) \int\limits_{a}^{b}f\left( u,d\right)
du+\int\limits_{a}^{b}\int\limits_{c}^{d}f\left( u,v\right) dudv\right\vert
\\
&\leq &\frac{\left( b-a\right) \left( d-c\right) }{\left( s+2\right) ^{2}} \\
&&\times \left[ \frac{1}{\left( s+1\right) ^{2}}\left\vert \frac{\partial
^{2}f}{\partial t\partial \lambda }\left( a,c\right) \right\vert +\frac{1}{%
s+1}\left\vert \frac{\partial ^{2}f}{\partial t\partial \lambda }\left(
b,c\right) \right\vert +\frac{1}{s+1}\left\vert \frac{\partial ^{2}f}{%
\partial t\partial \lambda }\left( a,d\right) \right\vert +\left\vert \frac{%
\partial ^{2}f}{\partial t\partial \lambda }\left( b,d\right) \right\vert %
\right] .
\end{eqnarray*}%
(2) Under the assumptions of Theorem \ref{t1}, if we choose $x=b,$ $y=d,$ we
obtain the following inequality;%
\begin{eqnarray*}
&&\frac{1}{\left( b-a\right) \left( d-c\right) }\left\vert f\left(
a,c\right) -\left( b-a\right) \int\limits_{c}^{d}f\left( a,v\right)
dv-\left( d-c\right) \int\limits_{a}^{b}f\left( u,c\right)
du+\int\limits_{a}^{b}\int\limits_{c}^{d}f\left( u,v\right) dudv\right\vert
\\
&\leq &\frac{\left( b-a\right) \left( d-c\right) }{\left( s+2\right) ^{2}} \\
&&\times \left[ \frac{1}{\left( s+1\right) ^{2}}\left\vert \frac{\partial
^{2}f}{\partial t\partial \lambda }\left( b,d\right) \right\vert +\frac{1}{%
s+1}\left\vert \frac{\partial ^{2}f}{\partial t\partial \lambda }\left(
a,d\right) \right\vert +\frac{1}{s+1}\left\vert \frac{\partial ^{2}f}{%
\partial t\partial \lambda }\left( b,c\right) \right\vert +\left\vert \frac{%
\partial ^{2}f}{\partial t\partial \lambda }\left( a,c\right) \right\vert %
\right] .
\end{eqnarray*}%
(3)Under the assumptions of Theorem \ref{t1}, if we choose $x=a,$ $y=d,$ we
obtain the following inequality;%
\begin{eqnarray*}
&&\frac{1}{\left( b-a\right) \left( d-c\right) }\left\vert f\left(
b,c\right) -\left( b-a\right) \int\limits_{c}^{d}f\left( b,v\right)
dv-\left( d-c\right) \int\limits_{a}^{b}f\left( u,c\right)
du+\int\limits_{a}^{b}\int\limits_{c}^{d}f\left( u,v\right) dudv\right\vert
\\
&\leq &\frac{\left( b-a\right) \left( d-c\right) }{\left( s+2\right) ^{2}} \\
&&\times \left[ \frac{1}{\left( s+1\right) ^{2}}\left\vert \frac{\partial
^{2}f}{\partial t\partial \lambda }\left( a,d\right) \right\vert +\frac{1}{%
s+1}\left\vert \frac{\partial ^{2}f}{\partial t\partial \lambda }\left(
b,d\right) \right\vert +\frac{1}{s+1}\left\vert \frac{\partial ^{2}f}{%
\partial t\partial \lambda }\left( a,c\right) \right\vert +\left\vert \frac{%
\partial ^{2}f}{\partial t\partial \lambda }\left( b,c\right) \right\vert %
\right]
\end{eqnarray*}%
(4)Under the assumptions of Theorem \ref{t1}, if we choose $x=b,$ $y=c,$ we
obtain the following inequality;%
\begin{eqnarray*}
&&\frac{1}{\left( b-a\right) \left( d-c\right) }\left\vert f\left(
a,d\right) -\left( b-a\right) \int\limits_{c}^{d}f\left( a,v\right)
dv-\left( d-c\right) \int\limits_{a}^{b}f\left( u,d\right)
du+\int\limits_{a}^{b}\int\limits_{c}^{d}f\left( u,v\right) dudv\right\vert
\\
&\leq &\frac{\left( b-a\right) \left( d-c\right) }{\left( s+2\right) ^{2}} \\
&&\times \left[ \frac{1}{\left( s+1\right) ^{2}}\left\vert \frac{\partial
^{2}f}{\partial t\partial \lambda }\left( b,c\right) \right\vert +\frac{1}{%
s+1}\left\vert \frac{\partial ^{2}f}{\partial t\partial \lambda }\left(
a,c\right) \right\vert +\frac{1}{s+1}\left\vert \frac{\partial ^{2}f}{%
\partial t\partial \lambda }\left( b,d\right) \right\vert +\left\vert \frac{%
\partial ^{2}f}{\partial t\partial \lambda }\left( a,d\right) \right\vert %
\right]
\end{eqnarray*}
\end{corollary}

\begin{remark}
From sum of four inequalities above, we obtain;%
\begin{eqnarray}
&&\left\vert f\left( b,d\right) -\left( b-a\right)
\int\limits_{c}^{d}f\left( b,v\right) dv-\left( d-c\right)
\int\limits_{a}^{b}f\left( u,d\right)
du+\int\limits_{a}^{b}\int\limits_{c}^{d}f\left( u,v\right) dudv\right\vert
\label{c1.5} \\
&&+\left\vert f\left( a,c\right) -\left( b-a\right)
\int\limits_{c}^{d}f\left( a,v\right) dv-\left( d-c\right)
\int\limits_{a}^{b}f\left( u,c\right)
du+\int\limits_{a}^{b}\int\limits_{c}^{d}f\left( u,v\right) dudv\right\vert 
\notag \\
&&+\left\vert f\left( b,c\right) -\left( b-a\right)
\int\limits_{c}^{d}f\left( b,v\right) dv-\left( d-c\right)
\int\limits_{a}^{b}f\left( u,c\right)
du+\int\limits_{a}^{b}\int\limits_{c}^{d}f\left( u,v\right) dudv\right\vert 
\notag \\
&&+\left\vert f\left( a,d\right) -\left( b-a\right)
\int\limits_{c}^{d}f\left( a,v\right) dv-\left( d-c\right)
\int\limits_{a}^{b}f\left( u,d\right)
du+\int\limits_{a}^{b}\int\limits_{c}^{d}f\left( u,v\right) dudv\right\vert 
\notag \\
&\leq &\frac{\left( b-a\right) ^{2}\left( d-c\right) ^{2}}{\left( s+1\right)
^{2}}\times  \notag \\
&&\left[ \left\vert \frac{\partial ^{2}f}{\partial t\partial \lambda }\left(
b,d\right) \right\vert +\left\vert \frac{\partial ^{2}f}{\partial t\partial
\lambda }\left( a,c\right) \right\vert +\left\vert \frac{\partial ^{2}f}{%
\partial t\partial \lambda }\left( b,c\right) \right\vert +\left\vert \frac{%
\partial ^{2}f}{\partial t\partial \lambda }\left( a,d\right) \right\vert %
\right] .  \notag
\end{eqnarray}
\end{remark}

\begin{corollary}
Under the assumptions of Theorem \ref{t1}, if we choose $x=\frac{a+b}{2},$ $%
y=\frac{c+d}{2},$ we obtain the following inequality;%
\begin{eqnarray*}
&&\left\vert \frac{f\left( a,c\right) +f\left( a,d\right) +f\left(
b,c\right) +f\left( b,d\right) }{4\left( b-a\right) \left( d-c\right) }-%
\frac{1}{2\left( d-c\right) }\int\limits_{c}^{d}f\left( a,v\right) dv-\frac{1%
}{2\left( d-c\right) }\int\limits_{c}^{d}f\left( b,v\right) dv\right. \\
&&\left. -\frac{1}{2\left( b-a\right) }\int\limits_{a}^{b}f\left( u,d\right)
du-\frac{1}{2\left( b-a\right) }\int\limits_{a}^{b}f\left( u,c\right) du+%
\frac{1}{\left( b-a\right) \left( d-c\right) }\int\limits_{a}^{b}\int%
\limits_{c}^{d}f\left( u,v\right) dudv\right\vert \\
&\leq &\frac{\left( b-a\right) \left( d-c\right) }{4\left( s+2\right) ^{2}}%
\left[ \frac{1}{\left( s+1\right) ^{2}}\left\vert \frac{\partial ^{2}f}{%
\partial t\partial \lambda }\left( \frac{a+b}{2},\frac{c+d}{2}\right)
\right\vert \right. \\
&&+\frac{1}{2\left( s+1\right) }\left\{ \left\vert \frac{\partial ^{2}f}{%
\partial t\partial \lambda }\left( a,\frac{c+d}{2}\right) \right\vert
+\left\vert \frac{\partial ^{2}f}{\partial t\partial \lambda }\left( b,\frac{%
c+d}{2}\right) \right\vert \right\} \\
&&+\frac{1}{2\left( s+1\right) }\left\{ \left\vert \frac{\partial ^{2}f}{%
\partial t\partial \lambda }\left( \frac{a+b}{2},c\right) \right\vert
+\left\vert \frac{\partial ^{2}f}{\partial t\partial \lambda }\left( \frac{%
a+b}{2},d\right) \right\vert \right\} \\
&&\left. +\frac{1}{4}\left\{ \left\vert \frac{\partial ^{2}f}{\partial
t\partial \lambda }\left( a,c\right) \right\vert +\left\vert \frac{\partial
^{2}f}{\partial t\partial \lambda }\left( a,d\right) \right\vert +\left\vert 
\frac{\partial ^{2}f}{\partial t\partial \lambda }\left( b,c\right)
\right\vert +\left\vert \frac{\partial ^{2}f}{\partial t\partial \lambda }%
\left( b,d\right) \right\vert \right\} \right] .
\end{eqnarray*}
\end{corollary}

\bigskip

\begin{theorem}
\label{t2}\bigskip Let $f:\Delta =\left[ a,b\right] \times \left[ c,d\right]
\rightarrow 
\mathbb{R}
$ be a partial differentiable mapping on $\Delta =\left[ a,b\right] \times %
\left[ c,d\right] $ and $\frac{\partial ^{2}f}{\partial t\partial \lambda }%
\in L\left( \Delta \right) $. If $\left\vert \frac{\partial ^{2}f}{\partial
t\partial \lambda }\right\vert ^{q},$ $q>1,$ is a $s-$convex function in the
second sense on the co-ordinates on $\Delta ,$ for some fixed $s\in \left(
0,1\right] $, then the following inequality holds;%
\begin{eqnarray}
&&  \label{22} \\
&&\left\vert \frac{1}{\left( b-a\right) \left( d-c\right) }\left[ A-\left(
x-a\right) \int\limits_{c}^{d}f\left( a,v\right) dv-\left( b-x\right)
\int\limits_{c}^{d}f\left( b,v\right) dv\right. \right.  \notag \\
&&\left. \left. -\left( d-y\right) \int\limits_{a}^{b}f\left( u,d\right)
du-\left( y-c\right) \int\limits_{a}^{b}f\left( u,c\right)
du+\int\limits_{a}^{b}\int\limits_{c}^{d}f\left( u,v\right) dudv\right]
\right\vert  \notag \\
&\leq &\frac{1}{\left( p+1\right) ^{\frac{2}{p}}}\frac{1}{\left( s+1\right)
^{\frac{2}{q}}}\times  \notag \\
&&\left\{ \frac{\left( x-a\right) ^{2}\left( y-c\right) ^{2}}{\left(
b-a\right) \left( d-c\right) }\left( \left\vert \frac{\partial ^{2}f}{%
\partial t\partial \lambda }\left( x,y\right) \right\vert ^{q}+\left\vert 
\frac{\partial ^{2}f}{\partial t\partial \lambda }\left( x,c\right)
\right\vert ^{q}+\left\vert \frac{\partial ^{2}f}{\partial t\partial \lambda 
}\left( a,y\right) \right\vert ^{q}+\left\vert \frac{\partial ^{2}f}{%
\partial t\partial \lambda }\left( a,c\right) \right\vert ^{q}\right) ^{%
\frac{1}{q}}\right.  \notag \\
&&+\frac{\left( x-a\right) ^{2}\left( d-y\right) ^{2}}{\left( b-a\right)
\left( d-c\right) }\left( \left\vert \frac{\partial ^{2}f}{\partial
t\partial \lambda }\left( x,y\right) \right\vert ^{q}+\left\vert \frac{%
\partial ^{2}f}{\partial t\partial \lambda }\left( x,d\right) \right\vert
^{q}+\left\vert \frac{\partial ^{2}f}{\partial t\partial \lambda }\left(
a,y\right) \right\vert ^{q}+\left\vert \frac{\partial ^{2}f}{\partial
t\partial \lambda }\left( a,d\right) \right\vert ^{q}\right) ^{\frac{1}{q}} 
\notag \\
&&+\frac{\left( b-x\right) ^{2}\left( y-c\right) ^{2}}{\left( b-a\right)
\left( d-c\right) }\left( \left\vert \frac{\partial ^{2}f}{\partial
t\partial \lambda }\left( x,y\right) \right\vert ^{q}+\left\vert \frac{%
\partial ^{2}f}{\partial t\partial \lambda }\left( x,c\right) \right\vert
^{q}+\left\vert \frac{\partial ^{2}f}{\partial t\partial \lambda }\left(
b,y\right) \right\vert ^{q}+\left\vert \frac{\partial ^{2}f}{\partial
t\partial \lambda }\left( b,c\right) \right\vert ^{q}\right) ^{\frac{1}{q}} 
\notag \\
&&\left. +\frac{\left( b-x\right) ^{2}\left( d-y\right) ^{2}}{\left(
b-a\right) \left( d-c\right) }\left( \left\vert \frac{\partial ^{2}f}{%
\partial t\partial \lambda }\left( x,y\right) \right\vert ^{q}+\left\vert 
\frac{\partial ^{2}f}{\partial t\partial \lambda }\left( x,d\right)
\right\vert ^{q}+\left\vert \frac{\partial ^{2}f}{\partial t\partial \lambda 
}\left( b,y\right) \right\vert ^{q}+\left\vert \frac{\partial ^{2}f}{%
\partial t\partial \lambda }\left( b,d\right) \right\vert ^{q}\right) ^{%
\frac{1}{q}}\right\}  \notag
\end{eqnarray}%
where $p^{-1}+q^{-1}=1$.
\end{theorem}

\begin{proof}
From Lemma \ref{l1}, we have%
\begin{eqnarray*}
&&\left\vert \frac{1}{\left( b-a\right) \left( d-c\right) }\left[ A-\left(
x-a\right) \int\limits_{c}^{d}f\left( a,v\right) dv-\left( b-x\right)
\int\limits_{c}^{d}f\left( b,v\right) dv\right. \right. \\
&&\left. \left. -\left( d-y\right) \int\limits_{a}^{b}f\left( u,d\right)
du-\left( y-c\right) \int\limits_{a}^{b}f\left( u,c\right)
du+\int\limits_{a}^{b}\int\limits_{c}^{d}f\left( u,v\right) dudv\right]
\right\vert \\
&\leq &\frac{\left( x-a\right) ^{2}\left( y-c\right) ^{2}}{\left( b-a\right)
\left( d-c\right) }\int\limits_{0}^{1}\int\limits_{0}^{1}\left\vert \left(
t-1\right) \left( \lambda -1\right) \right\vert \left\vert \frac{\partial
^{2}f}{\partial t\partial \lambda }\left( tx+\left( 1-t\right) a,\lambda
y+\left( 1-\lambda \right) c\right) \right\vert d\lambda dt \\
&&+\frac{\left( x-a\right) ^{2}\left( d-y\right) ^{2}}{\left( b-a\right)
\left( d-c\right) }\int\limits_{0}^{1}\int\limits_{0}^{1}\left\vert \left(
t-1\right) \left( 1-\lambda \right) \right\vert \left\vert \frac{\partial
^{2}f}{\partial t\partial \lambda }\left( tx+\left( 1-t\right) a,\lambda
y+\left( 1-\lambda \right) d\right) \right\vert d\lambda dt \\
&&+\frac{\left( b-x\right) ^{2}\left( y-c\right) ^{2}}{\left( b-a\right)
\left( d-c\right) }\int\limits_{0}^{1}\int\limits_{0}^{1}\left\vert \left(
1-t\right) \left( \lambda -1\right) \right\vert \left\vert \frac{\partial
^{2}f}{\partial t\partial \lambda }\left( tx+\left( 1-t\right) b,\lambda
y+\left( 1-\lambda \right) c\right) \right\vert d\lambda dt \\
&&+\frac{\left( b-x\right) ^{2}\left( d-y\right) ^{2}}{\left( b-a\right)
\left( d-c\right) }\int\limits_{0}^{1}\int\limits_{0}^{1}\left\vert \left(
1-t\right) \left( 1-\lambda \right) \right\vert \left\vert \frac{\partial
^{2}f}{\partial t\partial \lambda }\left( tx+\left( 1-t\right) b,\lambda
y+\left( 1-\lambda \right) d\right) \right\vert d\lambda dt.
\end{eqnarray*}%
By using the well known H\"{o}lder inequality for double integrals, then one
has:%
\begin{eqnarray}
&&\left\vert \frac{1}{\left( b-a\right) \left( d-c\right) }\left[ A-\left(
x-a\right) \int\limits_{c}^{d}f\left( a,v\right) dv-\left( b-x\right)
\int\limits_{c}^{d}f\left( b,v\right) dv\right. \right.  \label{m1} \\
&&\left. \left. -\left( d-y\right) \int\limits_{a}^{b}f\left( u,d\right)
du-\left( y-c\right) \int\limits_{a}^{b}f\left( u,c\right)
du+\int\limits_{a}^{b}\int\limits_{c}^{d}f\left( u,v\right) dudv\right]
\right\vert  \notag \\
&\leq &\frac{\left( x-a\right) ^{2}\left( y-c\right) ^{2}}{\left( b-a\right)
\left( d-c\right) }\left( \int\limits_{0}^{1}\int\limits_{0}^{1}\left\vert
\left( t-1\right) \left( \lambda -1\right) \right\vert ^{p}d\lambda
dt\right) ^{\frac{1}{p}}  \notag \\
&&\left( \int\limits_{0}^{1}\int\limits_{0}^{1}\left\vert \frac{\partial
^{2}f}{\partial t\partial \lambda }\left( tx+\left( 1-t\right) a,\lambda
y+\left( 1-\lambda \right) c\right) \right\vert ^{q}d\lambda dt\right) ^{%
\frac{1}{q}}  \notag \\
&&+\frac{\left( x-a\right) ^{2}\left( d-y\right) ^{2}}{\left( b-a\right)
\left( d-c\right) }\left( \int\limits_{0}^{1}\int\limits_{0}^{1}\left\vert
\left( t-1\right) \left( 1-\lambda \right) \right\vert ^{p}d\lambda
dt\right) ^{\frac{1}{p}}  \notag \\
&&\left( \int\limits_{0}^{1}\int\limits_{0}^{1}\left\vert \frac{\partial
^{2}f}{\partial t\partial \lambda }\left( tx+\left( 1-t\right) a,\lambda
y+\left( 1-\lambda \right) d\right) \right\vert ^{q}d\lambda dt\right) ^{%
\frac{1}{q}}  \notag \\
&&+\frac{\left( b-x\right) ^{2}\left( y-c\right) ^{2}}{\left( b-a\right)
\left( d-c\right) }\left( \int\limits_{0}^{1}\int\limits_{0}^{1}\left\vert
\left( 1-t\right) \left( \lambda -1\right) \right\vert ^{p}d\lambda
dt\right) ^{\frac{1}{p}}  \notag \\
&&\left( \int\limits_{0}^{1}\int\limits_{0}^{1}\left\vert \frac{\partial
^{2}f}{\partial t\partial \lambda }\left( tx+\left( 1-t\right) b,\lambda
y+\left( 1-\lambda \right) c\right) \right\vert ^{q}d\lambda dt\right) ^{%
\frac{1}{q}}  \notag \\
&&+\frac{\left( b-x\right) ^{2}\left( d-y\right) ^{2}}{\left( b-a\right)
\left( d-c\right) }\left( \int\limits_{0}^{1}\int\limits_{0}^{1}\left\vert
\left( 1-t\right) \left( 1-\lambda \right) \right\vert ^{p}d\lambda
dt\right) ^{\frac{1}{p}}  \notag \\
&&\left( \int\limits_{0}^{1}\int\limits_{0}^{1}\left\vert \frac{\partial
^{2}f}{\partial t\partial \lambda }\left( tx+\left( 1-t\right) b,\lambda
y+\left( 1-\lambda \right) d\right) \right\vert ^{q}d\lambda dt\right) ^{%
\frac{1}{q}}  \notag
\end{eqnarray}

Since $\left\vert \frac{\partial ^{2}f}{\partial t\partial \lambda }%
\right\vert ^{q},q>1,$ is $s-$convex function in the second sense on the
co-ordinates on $\Delta $, for some fixed $s\in \left( 0,1\right] ,$ we know
that for $t\in \left[ 0,1\right] $%
\begin{eqnarray*}
&&\left\vert \frac{\partial ^{2}f}{\partial t\partial \lambda }\left(
tx+\left( 1-t\right) a,\lambda y+\left( 1-\lambda \right) c\right)
\right\vert ^{q} \\
&\leq &t^{s}\left\vert \frac{\partial ^{2}f}{\partial t\partial \lambda }%
\left( x,\lambda y+\left( 1-\lambda \right) c\right) \right\vert ^{q}+\left(
1-t\right) ^{s}\left\vert \frac{\partial ^{2}f}{\partial t\partial \lambda }%
\left( a,\lambda y+\left( 1-\lambda \right) c\right) \right\vert ^{q} \\
&\leq &t^{s}\left( \lambda ^{s}\left\vert \frac{\partial ^{2}f}{\partial
t\partial \lambda }\left( x,y\right) \right\vert ^{q}+\left( 1-\lambda
\right) ^{s}\left\vert \frac{\partial ^{2}f}{\partial t\partial \lambda }%
\left( x,c\right) \right\vert ^{q}\right) \\
&&+\left( 1-t\right) ^{s}\left( \lambda ^{s}\left\vert \frac{\partial ^{2}f}{%
\partial t\partial \lambda }\left( a,y\right) \right\vert ^{q}+\left(
1-\lambda \right) ^{s}\left\vert \frac{\partial ^{2}f}{\partial t\partial
\lambda }\left( a,c\right) \right\vert ^{q}\right)
\end{eqnarray*}%
hence, it follows that%
\begin{eqnarray}
&&  \label{m2} \\
&&\left( \int\limits_{0}^{1}\int\limits_{0}^{1}\left\vert \frac{\partial
^{2}f}{\partial t\partial \lambda }\left( tx+\left( 1-t\right) a,\lambda
y+\left( 1-\lambda \right) c\right) \right\vert ^{q}d\lambda dt\right) ^{%
\frac{1}{q}}  \notag \\
&\leq &\left( \int\limits_{0}^{1}\int\limits_{0}^{1}\left\{ t^{s}\lambda
^{s}\left\vert \frac{\partial ^{2}f}{\partial t\partial \lambda }\left(
x,y\right) \right\vert ^{q}+t^{s}\left( 1-\lambda \right) ^{s}\left\vert 
\frac{\partial ^{2}f}{\partial t\partial \lambda }\left( x,c\right)
\right\vert ^{q}\right. \right.  \notag \\
&&\left. \left. +\left( 1-t\right) ^{s}\lambda ^{s}\left\vert \frac{\partial
^{2}f}{\partial t\partial \lambda }\left( a,y\right) \right\vert ^{q}+\left(
1-t\right) ^{s}\left( 1-\lambda \right) ^{s}\left\vert \frac{\partial ^{2}f}{%
\partial t\partial \lambda }\left( a,c\right) \right\vert ^{q}\right\}
dtd\lambda \right) ^{\frac{1}{q}}  \notag \\
&=&\frac{1}{\left( s+1\right) ^{\frac{2}{q}}}\left( \left\vert \frac{%
\partial ^{2}f}{\partial t\partial \lambda }\left( x,y\right) \right\vert
^{q}+\left\vert \frac{\partial ^{2}f}{\partial t\partial \lambda }\left(
x,c\right) \right\vert ^{q}+\left\vert \frac{\partial ^{2}f}{\partial
t\partial \lambda }\left( a,y\right) \right\vert ^{q}+\left\vert \frac{%
\partial ^{2}f}{\partial t\partial \lambda }\left( a,c\right) \right\vert
^{q}\right) ^{\frac{1}{q}}  \notag
\end{eqnarray}%
A similar way for other integral, since $\left\vert \frac{\partial ^{2}f}{%
\partial t\partial \lambda }\right\vert ^{q},q>1,$ is co-ordinated $s-$%
convex function on $\Delta $, we get%
\begin{eqnarray}
&&  \label{m3} \\
&&\left( \int\limits_{0}^{1}\int\limits_{0}^{1}\left\vert \frac{\partial
^{2}f}{\partial t\partial \lambda }\left( tx+\left( 1-t\right) a,\lambda
y+\left( 1-\lambda \right) d\right) \right\vert ^{q}dsdt\right) ^{\frac{1}{q}%
}  \notag \\
&\leq &\frac{1}{\left( s+1\right) ^{\frac{2}{q}}}\left( \left\vert \frac{%
\partial ^{2}f}{\partial t\partial \lambda }\left( x,y\right) \right\vert
^{q}+\left\vert \frac{\partial ^{2}f}{\partial t\partial \lambda }\left(
x,d\right) \right\vert ^{q}+\left\vert \frac{\partial ^{2}f}{\partial
t\partial \lambda }\left( a,y\right) \right\vert ^{q}+\left\vert \frac{%
\partial ^{2}f}{\partial t\partial \lambda }\left( a,d\right) \right\vert
^{q}\right) ^{\frac{1}{q}},  \notag
\end{eqnarray}%
\begin{eqnarray}
&&  \label{m4} \\
&&\left( \int\limits_{0}^{1}\int\limits_{0}^{1}\left\vert \frac{\partial
^{2}f}{\partial t\partial s}\left( tx+\left( 1-t\right) b,\lambda y+\left(
1-\lambda \right) c\right) \right\vert ^{q}dsdt\right) ^{\frac{1}{q}}  \notag
\\
&\leq &\frac{1}{\left( s+1\right) ^{\frac{2}{q}}}\left( \left\vert \frac{%
\partial ^{2}f}{\partial t\partial \lambda }\left( x,y\right) \right\vert
^{q}+\left\vert \frac{\partial ^{2}f}{\partial t\partial \lambda }\left(
x,c\right) \right\vert ^{q}+\left\vert \frac{\partial ^{2}f}{\partial
t\partial \lambda }\left( b,y\right) \right\vert ^{q}+\left\vert \frac{%
\partial ^{2}f}{\partial t\partial \lambda }\left( b,c\right) \right\vert
^{q}\right) ^{\frac{1}{q}},  \notag
\end{eqnarray}%
\begin{eqnarray}
&&  \label{m5} \\
&&\left( \int\limits_{0}^{1}\int\limits_{0}^{1}\left\vert \frac{\partial
^{2}f}{\partial t\partial \lambda }\left( tx+\left( 1-t\right) b,\lambda
y+\left( 1-\lambda \right) d\right) \right\vert ^{q}dsdt\right) ^{\frac{1}{q}%
}  \notag \\
&\leq &\frac{1}{\left( s+1\right) ^{\frac{2}{q}}}\left( \left\vert \frac{%
\partial ^{2}f}{\partial t\partial \lambda }\left( x,y\right) \right\vert
^{q}+\left\vert \frac{\partial ^{2}f}{\partial t\partial \lambda }\left(
x,d\right) \right\vert ^{q}+\left\vert \frac{\partial ^{2}f}{\partial
t\partial \lambda }\left( b,y\right) \right\vert ^{q}+\left\vert \frac{%
\partial ^{2}f}{\partial t\partial \lambda }\left( b,d\right) \right\vert
^{q}\right) ^{\frac{1}{q}}.  \notag
\end{eqnarray}%
By the (\ref{m2})-(\ref{m5}), we get the inequality (\ref{22}).
\end{proof}

\begin{corollary}
\bigskip \label{c2}(1) Under the assumptions of Theorem \ref{t2}, if we
choose $x=a,$ $y=c,$ or $x=b,$ $y=d,$ we obtain the following inequality;%
\begin{eqnarray}
&&  \label{mt1} \\
&&\frac{1}{\left( b-a\right) \left( d-c\right) }\left\vert f\left(
b,d\right) -\left( b-a\right) \int\limits_{c}^{d}f\left( b,v\right)
dv-\left( d-c\right) \int_{a}^{b}f\left( u,d\right)
du+\int\limits_{a}^{b}\int\limits_{c}^{d}f\left( u,v\right) dudv\right\vert 
\notag \\
&\leq &\frac{\left( b-a\right) \left( d-c\right) }{\left( p+1\right) ^{\frac{%
2}{p}}\left( s+1\right) ^{\frac{2}{q}}}  \notag \\
&&\left( \left\vert \frac{\partial ^{2}f}{\partial t\partial \lambda }\left(
a,c\right) \right\vert ^{q}+\left\vert \frac{\partial ^{2}f}{\partial
t\partial \lambda }\left( a,d\right) \right\vert ^{q}+\left\vert \frac{%
\partial ^{2}f}{\partial t\partial \lambda }\left( b,c\right) \right\vert
^{q}+\left\vert \frac{\partial ^{2}f}{\partial t\partial \lambda }\left(
b,d\right) \right\vert ^{q}\right) ^{\frac{1}{q}}  \notag
\end{eqnarray}%
(2)Under the assumptions of Theorem \ref{t2}, if we choose $x=b,$ $y=d,$ we
obtain the following inequality;%
\begin{eqnarray}
&&  \label{mt2} \\
&&\frac{1}{\left( b-a\right) \left( d-c\right) }\left\vert f\left(
a,c\right) -\left( b-a\right) \int\limits_{c}^{d}f\left( a,v\right)
dv-\left( d-c\right) \int\limits_{a}^{b}f\left( u,c\right)
du+\int\limits_{a}^{b}\int\limits_{c}^{d}f\left( u,v\right) dudv\right\vert 
\notag \\
&\leq &\frac{\left( b-a\right) \left( d-c\right) }{\left( p+1\right) ^{\frac{%
2}{p}}\left( s+1\right) ^{\frac{2}{q}}}  \notag \\
&&\left( \left\vert \frac{\partial ^{2}f}{\partial t\partial \lambda }\left(
b,d\right) \right\vert ^{q}+\left\vert \frac{\partial ^{2}f}{\partial
t\partial \lambda }\left( b,c\right) \right\vert ^{q}+\left\vert \frac{%
\partial ^{2}f}{\partial t\partial \lambda }\left( a,d\right) \right\vert
^{q}+\left\vert \frac{\partial ^{2}f}{\partial t\partial \lambda }\left(
a,c\right) \right\vert ^{q}\right) ^{\frac{1}{q}}  \notag
\end{eqnarray}%
(3)Under the assumptions of Theorem \ref{t2}, if we choose $x=a,$ $y=d,$ we
obtain the following inequality;%
\begin{eqnarray}
&&  \label{mt3} \\
&&\frac{1}{\left( b-a\right) \left( d-c\right) }\left\vert f\left(
b,c\right) -\left( b-a\right) \int\limits_{c}^{d}f\left( b,v\right)
dv-\left( d-c\right) \int\limits_{a}^{b}f\left( u,c\right)
du+\int\limits_{a}^{b}\int\limits_{c}^{d}f\left( u,v\right) dudv\right\vert 
\notag \\
&\leq &\frac{\left( b-a\right) \left( d-c\right) }{\left( p+1\right) ^{\frac{%
2}{p}}\left( s+1\right) ^{\frac{2}{q}}}  \notag \\
&&\left( \left\vert \frac{\partial ^{2}f}{\partial t\partial \lambda }\left(
a,d\right) \right\vert ^{q}+\left\vert \frac{\partial ^{2}f}{\partial
t\partial \lambda }\left( a,c\right) \right\vert ^{q}+\left\vert \frac{%
\partial ^{2}f}{\partial t\partial \lambda }\left( b,d\right) \right\vert
^{q}+\left\vert \frac{\partial ^{2}f}{\partial t\partial \lambda }\left(
b,c\right) \right\vert ^{q}\right) ^{\frac{1}{q}}  \notag
\end{eqnarray}%
(4)Under the assumptions of Theorem \ref{t2}, if we choose $x=b,$ $y=c,$ we
obtain the following inequality;%
\begin{eqnarray}
&&  \label{mt4} \\
&&\frac{1}{\left( b-a\right) \left( d-c\right) }\left\vert f\left(
a,d\right) -\left( b-a\right) \int\limits_{c}^{d}f\left( a,v\right)
dv-\left( d-c\right) \int\limits_{a}^{b}f\left( u,d\right)
du+\int\limits_{a}^{b}\int\limits_{c}^{d}f\left( u,v\right) dudv\right\vert 
\notag \\
&\leq &\frac{\left( b-a\right) \left( d-c\right) }{\left( p+1\right) ^{\frac{%
2}{p}}\left( s+1\right) ^{\frac{2}{q}}}  \notag \\
&&\left( \left\vert \frac{\partial ^{2}f}{\partial t\partial \lambda }\left(
b,c\right) \right\vert ^{q}+\left\vert \frac{\partial ^{2}f}{\partial
t\partial \lambda }\left( b,d\right) \right\vert ^{q}+\left\vert \frac{%
\partial ^{2}f}{\partial t\partial \lambda }\left( a,c\right) \right\vert
^{q}+\left\vert \frac{\partial ^{2}f}{\partial t\partial \lambda }\left(
a,d\right) \right\vert ^{q}\right) ^{\frac{1}{q}}  \notag
\end{eqnarray}%
(5) Under the assumptions of Theorem \ref{t2}, if we choose $x=\frac{a+b}{2}%
, $ $y=\frac{c+d}{2},$ we obtain the following inequality;%
\begin{eqnarray*}
&&\left\vert \frac{f\left( a,c\right) +f\left( a,d\right) +f\left(
b,c\right) +f\left( b,d\right) }{4\left( b-a\right) \left( d-c\right) }-%
\frac{1}{2\left( d-c\right) }\int\limits_{c}^{d}f\left( a,v\right) dv-\frac{1%
}{2\left( d-c\right) }\int\limits_{c}^{d}f\left( b,v\right) dv\right. \\
&&\left. -\frac{1}{2\left( b-a\right) }\int\limits_{a}^{b}f\left( u,d\right)
du-\frac{1}{2\left( b-a\right) }\int\limits_{a}^{b}f\left( u,c\right) du+%
\frac{1}{\left( b-a\right) \left( d-c\right) }\int\limits_{a}^{b}\int%
\limits_{c}^{d}f\left( u,v\right) dudv\right\vert \\
&\leq &\frac{\left( b-a\right) \left( d-c\right) }{16\left( p+1\right) ^{%
\frac{2}{p}}\left( s+1\right) ^{\frac{2}{q}}}\times \\
&&\left\{ \left( \left\vert \frac{\partial ^{2}f}{\partial t\partial \lambda 
}\left( \frac{a+b}{2},\frac{c+d}{2}\right) \right\vert ^{q}+\left\vert \frac{%
\partial ^{2}f}{\partial t\partial \lambda }\left( \frac{a+b}{2},c\right)
\right\vert ^{q}+\left\vert \frac{\partial ^{2}f}{\partial t\partial \lambda 
}\left( a,\frac{c+d}{2}\right) \right\vert ^{q}+\left\vert \frac{\partial
^{2}f}{\partial t\partial \lambda }\left( a,c\right) \right\vert ^{q}\right)
^{\frac{1}{q}}\right. \\
&&\left. +\left( \left\vert \frac{\partial ^{2}f}{\partial t\partial \lambda 
}\left( \frac{a+b}{2},\frac{c+d}{2}\right) \right\vert ^{q}+\left\vert \frac{%
\partial ^{2}f}{\partial t\partial \lambda }\left( \frac{a+b}{2},d\right)
\right\vert ^{q}+\left\vert \frac{\partial ^{2}f}{\partial t\partial \lambda 
}\left( a,\frac{c+d}{2}\right) \right\vert ^{q}+\left\vert \frac{\partial
^{2}f}{\partial t\partial \lambda }\left( a,d\right) \right\vert ^{q}\right)
^{\frac{1}{q}}\right. \\
&&\left. +\left( \left\vert \frac{\partial ^{2}f}{\partial t\partial \lambda 
}\left( \frac{a+b}{2},\frac{c+d}{2}\right) \right\vert ^{q}+\left\vert \frac{%
\partial ^{2}f}{\partial t\partial \lambda }\left( \frac{a+b}{2},c\right)
\right\vert ^{q}+\left\vert \frac{\partial ^{2}f}{\partial t\partial \lambda 
}\left( b,\frac{c+d}{2}\right) \right\vert ^{q}+\left\vert \frac{\partial
^{2}f}{\partial t\partial \lambda }\left( b,c\right) \right\vert ^{q}\right)
^{\frac{1}{q}}\right. \\
&&\left. +\left( \left\vert \frac{\partial ^{2}f}{\partial t\partial \lambda 
}\left( \frac{a+b}{2},\frac{c+d}{2}\right) \right\vert ^{q}+\left\vert \frac{%
\partial ^{2}f}{\partial t\partial \lambda }\left( \frac{a+b}{2},d\right)
\right\vert ^{q}+\left\vert \frac{\partial ^{2}f}{\partial t\partial \lambda 
}\left( b,\frac{c+d}{2}\right) \right\vert ^{q}+\left\vert \frac{\partial
^{2}f}{\partial t\partial \lambda }\left( b,d\right) \right\vert ^{q}\right)
^{\frac{1}{q}}\right\} .
\end{eqnarray*}
\end{corollary}

\begin{remark}
\bigskip From sum of (\ref{mt1})-(\ref{mt4}), we obtain;%
\begin{eqnarray}
&&  \label{metu} \\
&&\left\vert f\left( a,c\right) -\left( b-a\right)
\int\limits_{c}^{d}f\left( a,v\right) dv-\left( d-c\right)
\int\limits_{a}^{b}f\left( u,c\right)
du+\int\limits_{a}^{b}\int\limits_{c}^{d}f\left( u,v\right) dudv\right\vert 
\notag \\
&&+\left. \left\vert f\left( a,d\right) -\left( b-a\right)
\int\limits_{c}^{d}f\left( a,v\right) dv-\left( d-c\right)
\int\limits_{a}^{b}f\left( u,d\right)
du+\int\limits_{a}^{b}\int\limits_{c}^{d}f\left( u,v\right) dudv\right\vert
\right.  \notag \\
&&+\left. \left\vert f\left( b,c\right) -\left( b-a\right)
\int\limits_{c}^{d}f\left( b,v\right) dv-\left( d-c\right)
\int\limits_{a}^{b}f\left( u,c\right)
du+\int\limits_{a}^{b}\int\limits_{c}^{d}f\left( u,v\right) dudv\right\vert
\right.  \notag \\
&&+\left\vert f\left( b,d\right) -\left( b-a\right)
\int\limits_{c}^{d}f\left( b,v\right) dv-\left( d-c\right)
\int_{a}^{b}f\left( u,d\right)
du+\int\limits_{a}^{b}\int\limits_{c}^{d}f\left( u,v\right) dudv\right\vert 
\notag \\
&\leq &\frac{4\left( b-a\right) ^{2}\left( d-c\right) ^{2}}{\left(
p+1\right) ^{\frac{2}{p}}\left( s+1\right) ^{\frac{2}{q}}}\times  \notag \\
&&\left( \left\vert \frac{\partial ^{2}f}{\partial t\partial \lambda }\left(
a,c\right) \right\vert ^{q}+\left\vert \frac{\partial ^{2}f}{\partial
t\partial \lambda }\left( a,d\right) \right\vert ^{q}+\left\vert \frac{%
\partial ^{2}f}{\partial t\partial \lambda }\left( b,c\right) \right\vert
^{q}+\left\vert \frac{\partial ^{2}f}{\partial t\partial \lambda }\left(
b,d\right) \right\vert ^{q}\right) ^{\frac{1}{q}}.  \notag
\end{eqnarray}
\end{remark}

\begin{theorem}
\label{t3}\bigskip Let $f:\Delta =\left[ a,b\right] \times \left[ c,d\right]
\rightarrow 
\mathbb{R}
$ be a partial differentiable mapping on $\Delta =\left[ a,b\right] \times %
\left[ c,d\right] $ and $\frac{\partial ^{2}f}{\partial t\partial \lambda }%
\in L\left( \Delta \right) $. If $\left\vert \frac{\partial ^{2}f}{\partial
t\partial \lambda }\right\vert ^{q},$ $q\geq 1,$ is a $s-$convex function in
the second sense on the co-ordinates on $\Delta ,$ for some fixed $s\in
\left( 0,1\right] $, then the following inequality holds;%
\begin{eqnarray}
&&  \label{230} \\
&&\left\vert \frac{1}{\left( b-a\right) \left( d-c\right) }\left[ A-\left(
x-a\right) \int\limits_{c}^{d}f\left( a,v\right) dv-\left( b-x\right)
\int\limits_{c}^{d}f\left( b,v\right) dv\right. \right.   \notag \\
&&\left. \left. -\left( d-y\right) \int\limits_{a}^{b}f\left( u,d\right)
du-\left( y-c\right) \int\limits_{a}^{b}f\left( u,c\right)
du+\int\limits_{a}^{b}\int\limits_{c}^{d}f\left( u,v\right) dudv\right]
\right\vert   \notag \\
&\leq &\frac{2^{2-\frac{2}{q}}}{\left( s+1\right) ^{\frac{2}{q}}\left(
s+2\right) ^{\frac{2}{q}}}\times   \notag \\
&&\left\{ \frac{\left( x-a\right) ^{2}\left( y-c\right) ^{2}}{\left(
b-a\right) \left( d-c\right) }\left\{ \left\vert \frac{\partial ^{2}f}{%
\partial t\partial \lambda }\left( x,y\right) \right\vert ^{q}+\left(
s+1\right) \left\vert \frac{\partial ^{2}f}{\partial t\partial \lambda }%
\left( x,c\right) \right\vert ^{q}\right. \right.   \notag \\
&&+\left. \left( s+1\right) \left\vert \frac{\partial ^{2}f}{\partial
t\partial \lambda }\left( a,y\right) \right\vert ^{q}+\left( s+1\right)
^{2}\left\vert \frac{\partial ^{2}f}{\partial t\partial \lambda }\left(
a,c\right) \right\vert ^{q}\right\} ^{\frac{1}{q}}  \notag
\end{eqnarray}%
\begin{eqnarray*}
&& \\
&&+\frac{\left( x-a\right) ^{2}\left( d-y\right) ^{2}}{\left( b-a\right)
\left( d-c\right) }\left\{ \left\vert \frac{\partial ^{2}f}{\partial
t\partial \lambda }\left( x,y\right) \right\vert ^{q}+\left( s+1\right)
\left\vert \frac{\partial ^{2}f}{\partial t\partial \lambda }\left(
x,d\right) \right\vert ^{q}\right.  \\
&&+\left. \left( s+1\right) \left\vert \frac{\partial ^{2}f}{\partial
t\partial \lambda }\left( a,y\right) \right\vert ^{q}+\left( s+1\right)
^{2}\left\vert \frac{\partial ^{2}f}{\partial t\partial \lambda }\left(
a,d\right) \right\vert ^{q}\right\} ^{\frac{1}{q}} \\
&&+\frac{\left( b-x\right) ^{2}\left( y-c\right) ^{2}}{\left( b-a\right)
\left( d-c\right) }\left\{ \left\vert \frac{\partial ^{2}f}{\partial
t\partial \lambda }\left( x,y\right) \right\vert ^{q}+\left( s+1\right)
\left\vert \frac{\partial ^{2}f}{\partial t\partial \lambda }\left(
x,c\right) \right\vert ^{q}\right.  \\
&&+\left. \left( s+1\right) \left\vert \frac{\partial ^{2}f}{\partial
t\partial \lambda }\left( b,y\right) \right\vert ^{q}+\left( s+1\right)
^{2}\left\vert \frac{\partial ^{2}f}{\partial t\partial \lambda }\left(
b,c\right) \right\vert ^{q}\right\} ^{\frac{1}{q}} \\
&&+\frac{\left( b-x\right) ^{2}\left( d-y\right) ^{2}}{\left( b-a\right)
\left( d-c\right) }\left\{ \left\vert \frac{\partial ^{2}f}{\partial
t\partial \lambda }\left( x,y\right) \right\vert ^{q}+\left( s+1\right)
\left\vert \frac{\partial ^{2}f}{\partial t\partial \lambda }\left(
x,d\right) \right\vert ^{q}\right.  \\
&&\left. +\left. \left( s+1\right) \left\vert \frac{\partial ^{2}f}{\partial
t\partial \lambda }\left( b,y\right) \right\vert ^{q}+\left( s+1\right)
^{2}\left\vert \frac{\partial ^{2}f}{\partial t\partial \lambda }\left(
b,d\right) \right\vert ^{q}\right\} ^{\frac{1}{q}}\right\} .
\end{eqnarray*}
\end{theorem}

\begin{proof}
From Lemma \ref{l1}, we have%
\begin{eqnarray*}
&&\left\vert \frac{1}{\left( b-a\right) \left( d-c\right) }\left[ A-\left(
x-a\right) \int\limits_{c}^{d}f\left( a,v\right) dv-\left( b-x\right)
\int\limits_{c}^{d}f\left( b,v\right) dv\right. \right.  \\
&&\left. \left. -\left( d-y\right) \int\limits_{a}^{b}f\left( u,d\right)
du-\left( y-c\right) \int\limits_{a}^{b}f\left( u,c\right)
du+\int\limits_{a}^{b}\int\limits_{c}^{d}f\left( u,v\right) dudv\right]
\right\vert  \\
&\leq &\frac{\left( x-a\right) ^{2}\left( y-c\right) ^{2}}{\left( b-a\right)
\left( d-c\right) }\int\limits_{0}^{1}\int\limits_{0}^{1}\left\vert \left(
t-1\right) \left( \lambda -1\right) \right\vert \left\vert \frac{\partial
^{2}f}{\partial t\partial \lambda }\left( tx+\left( 1-t\right) a,\lambda
y+\left( 1-\lambda \right) c\right) \right\vert d\lambda dt \\
&&+\frac{\left( x-a\right) ^{2}\left( d-y\right) ^{2}}{\left( b-a\right)
\left( d-c\right) }\int\limits_{0}^{1}\int\limits_{0}^{1}\left\vert \left(
t-1\right) \left( 1-\lambda \right) \right\vert \left\vert \frac{\partial
^{2}f}{\partial t\partial \lambda }\left( tx+\left( 1-t\right) a,\lambda
y+\left( 1-\lambda \right) d\right) \right\vert d\lambda dt \\
&&+\frac{\left( b-x\right) ^{2}\left( y-c\right) ^{2}}{\left( b-a\right)
\left( d-c\right) }\int\limits_{0}^{1}\int\limits_{0}^{1}\left\vert \left(
1-t\right) \left( \lambda -1\right) \right\vert \left\vert \frac{\partial
^{2}f}{\partial t\partial \lambda }\left( tx+\left( 1-t\right) b,\lambda
y+\left( 1-\lambda \right) c\right) \right\vert d\lambda dt \\
&&+\frac{\left( b-x\right) ^{2}\left( d-y\right) ^{2}}{\left( b-a\right)
\left( d-c\right) }\int\limits_{0}^{1}\int\limits_{0}^{1}\left\vert \left(
1-t\right) \left( 1-\lambda \right) \right\vert \left\vert \frac{\partial
^{2}f}{\partial t\partial \lambda }\left( tx+\left( 1-t\right) b,\lambda
y+\left( 1-\lambda \right) d\right) \right\vert d\lambda dt.
\end{eqnarray*}%
By using the well known power mean inequality for double integrals, $%
f:\Delta \rightarrow 
\mathbb{R}
$ is co-ordinated $s-$convex on $\Delta ,$ then one has:%
\begin{eqnarray}
&&\left\vert \frac{1}{\left( b-a\right) \left( d-c\right) }\left[ A-\left(
x-a\right) \int\limits_{c}^{d}f\left( a,v\right) dv-\left( b-x\right)
\int\limits_{c}^{d}f\left( b,v\right) dv\right. \right.   \label{24} \\
&&\left. \left. -\left( d-y\right) \int\limits_{a}^{b}f\left( u,d\right)
du-\left( y-c\right) \int\limits_{a}^{b}f\left( u,c\right)
du+\int\limits_{a}^{b}\int\limits_{c}^{d}f\left( u,v\right) dudv\right]
\right\vert   \notag \\
&\leq &\frac{\left( x-a\right) ^{2}\left( y-c\right) ^{2}}{\left( b-a\right)
\left( d-c\right) }\left( \int\limits_{0}^{1}\int\limits_{0}^{1}\left\vert
\left( t-1\right) \left( \lambda -1\right) \right\vert d\lambda dt\right)
^{1-\frac{1}{q}}\times   \notag \\
&&\left( \int\limits_{0}^{1}\int\limits_{0}^{1}\left\vert \left( t-1\right)
\left( \lambda -1\right) \right\vert \left\vert \frac{\partial ^{2}f}{%
\partial t\partial \lambda }\left( tx+\left( 1-t\right) a,\lambda y+\left(
1-\lambda \right) c\right) \right\vert ^{q}d\lambda dt\right) ^{\frac{1}{q}}
\notag \\
&&+\frac{\left( x-a\right) ^{2}\left( d-y\right) ^{2}}{\left( b-a\right)
\left( d-c\right) }\left( \int\limits_{0}^{1}\int\limits_{0}^{1}\left\vert
\left( t-1\right) \left( 1-\lambda \right) \right\vert d\lambda dt\right)
^{1-\frac{1}{q}}\times   \notag \\
&&\left( \int\limits_{0}^{1}\int\limits_{0}^{1}\left\vert \left( t-1\right)
\left( 1-\lambda \right) \right\vert \left\vert \frac{\partial ^{2}f}{%
\partial t\partial \lambda }\left( tx+\left( 1-t\right) a,\lambda y+\left(
1-\lambda \right) d\right) \right\vert ^{q}d\lambda dt\right) ^{\frac{1}{q}}
\notag \\
&&+\frac{\left( b-x\right) ^{2}\left( y-c\right) ^{2}}{\left( b-a\right)
\left( d-c\right) }\left( \int\limits_{0}^{1}\int\limits_{0}^{1}\left\vert
\left( 1-t\right) \left( \lambda -1\right) \right\vert d\lambda dt\right)
^{1-\frac{1}{q}}\times   \notag \\
&&\left( \int\limits_{0}^{1}\int\limits_{0}^{1}\left\vert \left( 1-t\right)
\left( \lambda -1\right) \right\vert \left\vert \frac{\partial ^{2}f}{%
\partial t\partial \lambda }\left( tx+\left( 1-t\right) b,\lambda y+\left(
1-\lambda \right) c\right) \right\vert ^{q}d\lambda dt\right) ^{\frac{1}{q}}
\notag \\
&&+\frac{\left( b-x\right) ^{2}\left( d-y\right) ^{2}}{\left( b-a\right)
\left( d-c\right) }\left( \int\limits_{0}^{1}\int\limits_{0}^{1}\left\vert
\left( 1-t\right) \left( 1-\lambda \right) \right\vert d\lambda dt\right)
^{1-\frac{1}{q}}\times   \notag \\
&&\left( \int\limits_{0}^{1}\int\limits_{0}^{1}\left\vert \left( 1-t\right)
\left( 1-\lambda \right) \right\vert \left\vert \frac{\partial ^{2}f}{%
\partial t\partial \lambda }\left( tx+\left( 1-t\right) b,\lambda y+\left(
1-\lambda \right) d\right) \right\vert ^{q}d\lambda dt\right) ^{\frac{1}{q}}
\notag
\end{eqnarray}

Since $\left\vert \frac{\partial ^{2}f}{\partial t\partial \lambda }%
\right\vert ^{q}$ is $s-$convex function in the second sense on the
co-ordinates on $\Delta $, for some fixed $s\in \left( 0,1\right] ,$ we know
that for $t\in \left[ 0,1\right] $%
\begin{eqnarray*}
&&\left\vert \frac{\partial ^{2}f}{\partial t\partial \lambda }\left(
tx+\left( 1-t\right) a,\lambda y+\left( 1-\lambda \right) c\right)
\right\vert ^{q} \\
&\leq &t^{s}\left\vert \frac{\partial ^{2}f}{\partial t\partial \lambda }%
\left( x,\lambda y+\left( 1-\lambda \right) c\right) \right\vert ^{q}+\left(
1-t\right) ^{s}\left\vert \frac{\partial ^{2}f}{\partial t\partial \lambda }%
\left( a,\lambda y+\left( 1-\lambda \right) c\right) \right\vert ^{q}
\end{eqnarray*}%
and%
\begin{eqnarray*}
&&\left\vert \frac{\partial ^{2}f}{\partial t\partial \lambda }\left(
tx+\left( 1-t\right) a,\lambda y+\left( 1-\lambda \right) c\right)
\right\vert ^{q} \\
&\leq &t^{s}\lambda ^{s}\left\vert \frac{\partial ^{2}f}{\partial t\partial
\lambda }\left( x,y\right) \right\vert ^{q}+t^{s}\left( 1-\lambda \right)
^{s}\left\vert \frac{\partial ^{2}f}{\partial t\partial \lambda }\left(
x,c\right) \right\vert ^{q} \\
&&+\left( 1-t\right) ^{s}\lambda ^{s}\left\vert \frac{\partial ^{2}f}{%
\partial t\partial \lambda }\left( a,y\right) \right\vert ^{q}+\left(
1-t\right) ^{s}\left( 1-\lambda \right) ^{s}\left\vert \frac{\partial ^{2}f}{%
\partial t\partial \lambda }\left( a,c\right) \right\vert ^{q}
\end{eqnarray*}%
hence,\ it follows that%
\begin{eqnarray}
&&  \label{231} \\
&&\left( \int\limits_{0}^{1}\int\limits_{0}^{1}\left\vert \left( t-1\right)
\left( \lambda -1\right) \right\vert \left\vert \frac{\partial ^{2}f}{%
\partial t\partial \lambda }\left( tx+\left( 1-t\right) a,\lambda y+\left(
1-\lambda \right) c\right) \right\vert ^{q}d\lambda dt\right) ^{\frac{1}{q}}
\notag \\
&\leq &\left( \int\limits_{0}^{1}\int\limits_{0}^{1}\left\{ \left\vert
\left( t-1\right) \left( \lambda -1\right) \right\vert t^{s}\lambda
^{s}\left\vert \frac{\partial ^{2}f}{\partial t\partial \lambda }\left(
x,y\right) \right\vert ^{q}+\left\vert \left( t-1\right) \left( \lambda
-1\right) \right\vert t^{s}\left( 1-\lambda \right) ^{s}\left\vert \frac{%
\partial ^{2}f}{\partial t\partial \lambda }\left( x,c\right) \right\vert
^{q}\right. \right.  \notag \\
&&\left. \left. +\left\vert \left( t-1\right) \left( \lambda -1\right)
\right\vert \left( 1-t\right) ^{s}\lambda ^{s}\left\vert \frac{\partial ^{2}f%
}{\partial t\partial \lambda }\left( a,y\right) \right\vert ^{q}+\left\vert
\left( t-1\right) \left( \lambda -1\right) \right\vert \left( 1-t\right)
^{s}\left( 1-\lambda \right) ^{s}\left\vert \frac{\partial ^{2}f}{\partial
t\partial \lambda }\left( a,c\right) \right\vert ^{q}\right\} dtd\lambda
\right) ^{\frac{1}{q}}  \notag \\
&=&\left\{ \frac{1}{\left( s+1\right) ^{2}\left( s+2\right) ^{2}}\left\vert 
\frac{\partial ^{2}f}{\partial t\partial \lambda }\left( x,y\right)
\right\vert ^{q}+\frac{1}{\left( s+1\right) \left( s+2\right) ^{2}}%
\left\vert \frac{\partial ^{2}f}{\partial t\partial \lambda }\left(
x,c\right) \right\vert ^{q}\right.  \notag \\
&&+\left. \frac{1}{\left( s+1\right) \left( s+2\right) ^{2}}\left\vert \frac{%
\partial ^{2}f}{\partial t\partial \lambda }\left( a,y\right) \right\vert
^{q}+\frac{1}{\left( s+2\right) ^{2}}\left\vert \frac{\partial ^{2}f}{%
\partial t\partial \lambda }\left( a,c\right) \right\vert ^{q}\right\} ^{%
\frac{1}{q}}  \notag \\
&=&\frac{1}{\left( s+1\right) ^{\frac{2}{q}}\left( s+2\right) ^{\frac{2}{q}}}%
\left\{ \left\vert \frac{\partial ^{2}f}{\partial t\partial \lambda }\left(
x,y\right) \right\vert ^{q}+\left( s+1\right) \left\vert \frac{\partial ^{2}f%
}{\partial t\partial \lambda }\left( x,c\right) \right\vert ^{q}\right. 
\notag \\
&&+\left. \left( s+1\right) \left\vert \frac{\partial ^{2}f}{\partial
t\partial \lambda }\left( a,y\right) \right\vert ^{q}+\left( s+1\right)
^{2}\left\vert \frac{\partial ^{2}f}{\partial t\partial \lambda }\left(
a,c\right) \right\vert ^{q}\right\} ^{\frac{1}{q}}  \notag
\end{eqnarray}%
A similar way for other integral, since $\left\vert \frac{\partial ^{2}f}{%
\partial t\partial \lambda }\right\vert ^{q}$ is co-ordinated $s-$convex
function in the second sense on $\Delta $, we get%
\begin{eqnarray}
&&\left( \int\limits_{0}^{1}\int\limits_{0}^{1}\left\vert \left( t-1\right)
\left( 1-\lambda \right) \right\vert \left\vert \frac{\partial ^{2}f}{%
\partial t\partial \lambda }\left( tx+\left( 1-t\right) a,\lambda y+\left(
1-\lambda \right) d\right) \right\vert ^{q}dsdt\right) ^{\frac{1}{q}}
\label{232} \\
&\leq &\frac{1}{\left( s+1\right) ^{\frac{2}{q}}\left( s+2\right) ^{\frac{2}{%
q}}}\left\{ \left\vert \frac{\partial ^{2}f}{\partial t\partial \lambda }%
\left( x,y\right) \right\vert ^{q}+\left( s+1\right) \left\vert \frac{%
\partial ^{2}f}{\partial t\partial \lambda }\left( x,d\right) \right\vert
^{q}\right.  \notag \\
&&+\left. \left( s+1\right) \left\vert \frac{\partial ^{2}f}{\partial
t\partial \lambda }\left( a,y\right) \right\vert ^{q}+\left( s+1\right)
^{2}\left\vert \frac{\partial ^{2}f}{\partial t\partial \lambda }\left(
a,d\right) \right\vert ^{q}\right\} ^{\frac{1}{q}},  \notag
\end{eqnarray}%
\begin{eqnarray}
&&\left( \int\limits_{0}^{1}\int\limits_{0}^{1}\left\vert \left( 1-t\right)
\left( \lambda -1\right) \right\vert \left\vert \frac{\partial ^{2}f}{%
\partial t\partial s}\left( tx+\left( 1-t\right) b,\lambda y+\left(
1-\lambda \right) c\right) \right\vert ^{q}dsdt\right) ^{\frac{1}{q}}
\label{233} \\
&\leq &\frac{1}{\left( s+1\right) ^{\frac{2}{q}}\left( s+2\right) ^{\frac{2}{%
q}}}\left\{ \left\vert \frac{\partial ^{2}f}{\partial t\partial \lambda }%
\left( x,y\right) \right\vert ^{q}+\left( s+1\right) \left\vert \frac{%
\partial ^{2}f}{\partial t\partial \lambda }\left( x,c\right) \right\vert
^{q}\right.  \notag \\
&&+\left. \left( s+1\right) \left\vert \frac{\partial ^{2}f}{\partial
t\partial \lambda }\left( b,y\right) \right\vert ^{q}+\left( s+1\right)
^{2}\left\vert \frac{\partial ^{2}f}{\partial t\partial \lambda }\left(
b,c\right) \right\vert ^{q}\right\} ^{\frac{1}{q}},  \notag
\end{eqnarray}%
\begin{eqnarray}
&&\left( \int\limits_{0}^{1}\int\limits_{0}^{1}\left\vert \left( 1-t\right)
\left( 1-\lambda \right) \right\vert \left\vert \frac{\partial ^{2}f}{%
\partial t\partial \lambda }\left( tx+\left( 1-t\right) b,\lambda y+\left(
1-\lambda \right) d\right) \right\vert ^{q}dsdt\right) ^{\frac{1}{q}}
\label{234} \\
&\leq &\frac{1}{\left( s+1\right) ^{\frac{2}{q}}\left( s+2\right) ^{\frac{2}{%
q}}}\left\{ \left\vert \frac{\partial ^{2}f}{\partial t\partial \lambda }%
\left( x,y\right) \right\vert ^{q}+\left( s+1\right) \left\vert \frac{%
\partial ^{2}f}{\partial t\partial \lambda }\left( x,d\right) \right\vert
^{q}\right.  \notag \\
&&+\left. \left( s+1\right) \left\vert \frac{\partial ^{2}f}{\partial
t\partial \lambda }\left( b,y\right) \right\vert ^{q}+\left( s+1\right)
^{2}\left\vert \frac{\partial ^{2}f}{\partial t\partial \lambda }\left(
b,d\right) \right\vert ^{q}\right\} ^{\frac{1}{q}}.  \notag
\end{eqnarray}%
By the (\ref{231})-(\ref{234}), we get the inequality (\ref{230}).
\end{proof}

\begin{corollary}
\bigskip \label{c3}(1) Under the assumptions of Theorem \ref{t3}, if we
choose $x=a,$ $y=c,$ or $x=b,$ $y=d,$ we obtain the following inequality;%
\begin{eqnarray}
&&  \label{c31} \\
&&\frac{1}{\left( b-a\right) \left( d-c\right) }\left\vert f\left(
b,d\right) -\left( b-a\right) \int\limits_{c}^{d}f\left( b,v\right)
dv-\left( d-c\right) \int_{a}^{b}f\left( u,d\right)
du+\int\limits_{a}^{b}\int\limits_{c}^{d}f\left( u,v\right) dudv\right\vert 
\notag \\
&\leq &\frac{2^{2-\frac{2}{q}}\left( b-a\right) \left( d-c\right) }{\left(
s+1\right) ^{\frac{2}{q}}\left( s+2\right) ^{\frac{2}{q}}}\left\{ \left\vert 
\frac{\partial ^{2}f}{\partial t\partial \lambda }\left( a,c\right)
\right\vert ^{q}+\left( s+1\right) \left\vert \frac{\partial ^{2}f}{\partial
t\partial \lambda }\left( a,d\right) \right\vert ^{q}\right.  \notag \\
&&+\left. \left( s+1\right) \left\vert \frac{\partial ^{2}f}{\partial
t\partial \lambda }\left( b,c\right) \right\vert ^{q}+\left( s+1\right)
^{2}\left\vert \frac{\partial ^{2}f}{\partial t\partial \lambda }\left(
b,d\right) \right\vert ^{q}\right\} ^{\frac{1}{q}}.  \notag
\end{eqnarray}%
(2)Under the assumptions of Theorem \ref{t3}, if we choose $x=b,$ $y=d,$ we
obtain the following inequality;%
\begin{eqnarray}
&&  \label{c32} \\
&&\frac{1}{\left( b-a\right) \left( d-c\right) }\left\vert f\left(
a,c\right) -\left( b-a\right) \int\limits_{c}^{d}f\left( a,v\right)
dv-\left( d-c\right) \int\limits_{a}^{b}f\left( u,c\right)
du+\int\limits_{a}^{b}\int\limits_{c}^{d}f\left( u,v\right) dudv\right\vert 
\notag \\
&\leq &\frac{2^{2-\frac{2}{q}}\left( b-a\right) \left( d-c\right) }{\left(
s+1\right) ^{\frac{2}{q}}\left( s+2\right) ^{\frac{2}{q}}}\left\{ \left\vert 
\frac{\partial ^{2}f}{\partial t\partial \lambda }\left( b,d\right)
\right\vert ^{q}+\left( s+1\right) \left\vert \frac{\partial ^{2}f}{\partial
t\partial \lambda }\left( b,c\right) \right\vert ^{q}\right.  \notag \\
&&\left. +\left( s+1\right) \left\vert \frac{\partial ^{2}f}{\partial
t\partial \lambda }\left( a,d\right) \right\vert ^{q}+\left( s+1\right)
^{2}\left\vert \frac{\partial ^{2}f}{\partial t\partial \lambda }\left(
a,c\right) \right\vert ^{q}\right\} ^{\frac{1}{q}}  \notag
\end{eqnarray}%
(3)Under the assumptions of Theorem \ref{t3}, if we choose $x=a,$ $y=d,$ we
obtain the following inequality;%
\begin{eqnarray}
&&  \label{c33} \\
&&\frac{1}{\left( b-a\right) \left( d-c\right) }\left\vert f\left(
b,c\right) -\left( b-a\right) \int\limits_{c}^{d}f\left( b,v\right)
dv-\left( d-c\right) \int\limits_{a}^{b}f\left( u,c\right)
du+\int\limits_{a}^{b}\int\limits_{c}^{d}f\left( u,v\right) dudv\right\vert 
\notag \\
&\leq &\frac{2^{2-\frac{2}{q}}\left( b-a\right) \left( d-c\right) }{\left(
s+1\right) ^{\frac{2}{q}}\left( s+2\right) ^{\frac{2}{q}}}\left\{ \left\vert 
\frac{\partial ^{2}f}{\partial t\partial \lambda }\left( a,d\right)
\right\vert ^{q}+\left( s+1\right) \left\vert \frac{\partial ^{2}f}{\partial
t\partial \lambda }\left( a,c\right) \right\vert ^{q}\right.  \notag \\
&&\left. +\left( s+1\right) \left\vert \frac{\partial ^{2}f}{\partial
t\partial \lambda }\left( b,d\right) \right\vert ^{q}+\left( s+1\right)
^{2}\left\vert \frac{\partial ^{2}f}{\partial t\partial \lambda }\left(
b,c\right) \right\vert ^{q}\right\} ^{\frac{1}{q}}  \notag
\end{eqnarray}%
(4)Under the assumptions of Theorem \ref{t3}, if we choose $x=b,$ $y=c,$ we
obtain the following inequality;%
\begin{eqnarray}
&&  \label{c34} \\
&&\frac{1}{\left( b-a\right) \left( d-c\right) }\left\vert f\left(
a,d\right) -\left( b-a\right) \int\limits_{c}^{d}f\left( a,v\right)
dv-\left( d-c\right) \int\limits_{a}^{b}f\left( u,d\right)
du+\int\limits_{a}^{b}\int\limits_{c}^{d}f\left( u,v\right) dudv\right\vert 
\notag \\
&\leq &\frac{2^{2-\frac{2}{q}}\left( b-a\right) \left( d-c\right) }{\left(
s+1\right) ^{\frac{2}{q}}\left( s+2\right) ^{\frac{2}{q}}}\left\{ \left\vert 
\frac{\partial ^{2}f}{\partial t\partial \lambda }\left( b,c\right)
\right\vert ^{q}+\left( s+1\right) \left\vert \frac{\partial ^{2}f}{\partial
t\partial \lambda }\left( b,d\right) \right\vert ^{q}\right.  \notag \\
&&\left. +\left( s+1\right) \left\vert \frac{\partial ^{2}f}{\partial
t\partial \lambda }\left( a,c\right) \right\vert ^{q}+\left( s+1\right)
^{2}\left\vert \frac{\partial ^{2}f}{\partial t\partial \lambda }\left(
a,d\right) \right\vert ^{q}\right\} ^{\frac{1}{q}}  \notag
\end{eqnarray}%
(5) Under the assumptions of Theorem \ref{t3}, if we choose $x=\frac{a+b}{2}%
, $ $y=\frac{c+d}{2},$ we obtain the following inequality;%
\begin{eqnarray*}
&&\left\vert \frac{f\left( a,c\right) +f\left( a,d\right) +f\left(
b,c\right) +f\left( b,d\right) }{4\left( b-a\right) \left( d-c\right) }-%
\frac{1}{2\left( d-c\right) }\int\limits_{c}^{d}f\left( a,v\right) dv-\frac{1%
}{2\left( d-c\right) }\int\limits_{c}^{d}f\left( b,v\right) dv\right. \\
&&\left. -\frac{1}{2\left( b-a\right) }\int\limits_{a}^{b}f\left( u,d\right)
du-\frac{1}{2\left( b-a\right) }\int\limits_{a}^{b}f\left( u,c\right) du+%
\frac{1}{\left( b-a\right) \left( d-c\right) }\int\limits_{a}^{b}\int%
\limits_{c}^{d}f\left( u,v\right) dudv\right\vert \\
&\leq &\frac{\left( b-a\right) \left( d-c\right) }{4\left( 2\left(
s+1\right) \left( s+2\right) \right) ^{\frac{2}{q}}}\times \\
&&\left\{ \left( \left\vert \frac{\partial ^{2}f}{\partial t\partial \lambda 
}\left( \frac{a+b}{2},\frac{c+d}{2}\right) \right\vert ^{q}+\left(
s+1\right) \left\vert \frac{\partial ^{2}f}{\partial t\partial \lambda }%
\left( \frac{a+b}{2},c\right) \right\vert ^{q}\right. \right. \\
&&\left. +\left( s+1\right) \left\vert \frac{\partial ^{2}f}{\partial
t\partial \lambda }\left( a,\frac{c+d}{2}\right) \right\vert ^{q}+\left(
s+1\right) ^{2}\left\vert \frac{\partial ^{2}f}{\partial t\partial \lambda }%
\left( a,c\right) \right\vert ^{q}\right) ^{\frac{1}{q}} \\
&&+\left( \left\vert \frac{\partial ^{2}f}{\partial t\partial \lambda }%
\left( \frac{a+b}{2},\frac{c+d}{2}\right) \right\vert ^{q}+\left( s+1\right)
\left\vert \frac{\partial ^{2}f}{\partial t\partial \lambda }\left( \frac{a+b%
}{2},d\right) \right\vert ^{q}\right. \\
&&\left. +\left( s+1\right) \left\vert \frac{\partial ^{2}f}{\partial
t\partial \lambda }\left( a,\frac{c+d}{2}\right) \right\vert ^{q}+\left(
s+1\right) ^{2}\left\vert \frac{\partial ^{2}f}{\partial t\partial \lambda }%
\left( a,d\right) \right\vert ^{q}\right) ^{\frac{1}{q}} \\
&&+\left( \left\vert \frac{\partial ^{2}f}{\partial t\partial \lambda }%
\left( \frac{a+b}{2},\frac{c+d}{2}\right) \right\vert ^{q}+\left( s+1\right)
\left\vert \frac{\partial ^{2}f}{\partial t\partial \lambda }\left( \frac{a+b%
}{2},c\right) \right\vert ^{q}\right. \\
&&\left. +\left( s+1\right) \left\vert \frac{\partial ^{2}f}{\partial
t\partial \lambda }\left( b,\frac{c+d}{2}\right) \right\vert ^{q}+\left(
s+1\right) ^{2}\left\vert \frac{\partial ^{2}f}{\partial t\partial \lambda }%
\left( b,c\right) \right\vert ^{q}\right) ^{\frac{1}{q}} \\
&&+\left[ \left\vert \frac{\partial ^{2}f}{\partial t\partial \lambda }%
\left( \frac{a+b}{2},\frac{c+d}{2}\right) \right\vert ^{q}+\left( s+1\right)
\left\vert \frac{\partial ^{2}f}{\partial t\partial \lambda }\left( \frac{a+b%
}{2},d\right) \right\vert ^{q}\right. . \\
&&\left. \left. +\left( s+1\right) \left\vert \frac{\partial ^{2}f}{\partial
t\partial \lambda }\left( b,\frac{c+d}{2}\right) \right\vert ^{q}+\left(
s+1\right) ^{2}\left\vert \frac{\partial ^{2}f}{\partial t\partial \lambda }%
\left( b,d\right) \right\vert ^{q}\right) ^{\frac{1}{q}}\right\}
\end{eqnarray*}
\end{corollary}

\begin{remark}
\bigskip From sum of (\ref{c31})-(\ref{c34}), we get;%
\begin{eqnarray*}
&& \\
&&\left\vert f\left( a,c\right) -\left( b-a\right)
\int\limits_{c}^{d}f\left( a,v\right) dv-\left( d-c\right)
\int\limits_{a}^{b}f\left( u,c\right)
du+\int\limits_{a}^{b}\int\limits_{c}^{d}f\left( u,v\right) dudv\right\vert 
\\
&&+\left. \left\vert f\left( a,d\right) -\left( b-a\right)
\int\limits_{c}^{d}f\left( a,v\right) dv-\left( d-c\right)
\int\limits_{a}^{b}f\left( u,d\right)
du+\int\limits_{a}^{b}\int\limits_{c}^{d}f\left( u,v\right) dudv\right\vert
\right.  \\
&&+\left. \left\vert f\left( b,c\right) -\left( b-a\right)
\int\limits_{c}^{d}f\left( b,v\right) dv-\left( d-c\right)
\int\limits_{a}^{b}f\left( u,c\right)
du+\int\limits_{a}^{b}\int\limits_{c}^{d}f\left( u,v\right) dudv\right\vert
\right.  \\
&&+\left\vert f\left( b,d\right) -\left( b-a\right)
\int\limits_{c}^{d}f\left( b,v\right) dv-\left( d-c\right)
\int_{a}^{b}f\left( u,d\right)
du+\int\limits_{a}^{b}\int\limits_{c}^{d}f\left( u,v\right) dudv\right\vert 
\end{eqnarray*}%
\begin{eqnarray}
&\leq &\frac{4\left( b-a\right) ^{2}\left( d-c\right) ^{2}}{\left( 2\left(
s+1\right) \left( s+2\right) \right) ^{\frac{2}{q}}}\times   \notag \\
&&\left\{ \left( \left\vert \frac{\partial ^{2}f}{\partial t\partial \lambda 
}\left( a,c\right) \right\vert ^{q}+\left( s+1\right) \left\vert \frac{%
\partial ^{2}f}{\partial t\partial \lambda }\left( a,d\right) \right\vert
^{q}+\left( s+1\right) \left\vert \frac{\partial ^{2}f}{\partial t\partial
\lambda }\left( b,c\right) \right\vert ^{q}+\left( s+1\right) ^{2}\left\vert 
\frac{\partial ^{2}f}{\partial t\partial \lambda }\left( b,d\right)
\right\vert ^{q}\right) ^{\frac{1}{q}}\right.   \notag \\
&&+\left( \left( s+1\right) ^{2}\left\vert \frac{\partial ^{2}f}{\partial
t\partial \lambda }\left( a,c\right) \right\vert ^{q}+\left( s+1\right)
\left\vert \frac{\partial ^{2}f}{\partial t\partial \lambda }\left(
a,d\right) \right\vert ^{q}+\left( s+1\right) \left\vert \frac{\partial ^{2}f%
}{\partial t\partial \lambda }\left( b,c\right) \right\vert ^{q}+\left\vert 
\frac{\partial ^{2}f}{\partial t\partial \lambda }\left( b,d\right)
\right\vert ^{q}\right) ^{\frac{1}{q}}  \notag \\
&&+\left( \left( s+1\right) \left\vert \frac{\partial ^{2}f}{\partial
t\partial \lambda }\left( a,c\right) \right\vert ^{q}+\left\vert \frac{%
\partial ^{2}f}{\partial t\partial \lambda }\left( a,d\right) \right\vert
^{q}+\left( s+1\right) ^{2}\left\vert \frac{\partial ^{2}f}{\partial
t\partial \lambda }\left( b,c\right) \right\vert ^{q}+\left( s+1\right)
\left\vert \frac{\partial ^{2}f}{\partial t\partial \lambda }\left(
b,d\right) \right\vert ^{q}\right) ^{\frac{1}{q}}  \notag \\
&&\left. +\left( \left( s+1\right) \left\vert \frac{\partial ^{2}f}{\partial
t\partial \lambda }\left( a,c\right) \right\vert ^{q}+\left( s+1\right)
^{2}\left\vert \frac{\partial ^{2}f}{\partial t\partial \lambda }\left(
a,d\right) \right\vert ^{q}+\left\vert \frac{\partial ^{2}f}{\partial
t\partial \lambda }\left( b,c\right) \right\vert ^{q}+\left( s+1\right)
\left\vert \frac{\partial ^{2}f}{\partial t\partial \lambda }\left(
b,d\right) \right\vert ^{q}\right) ^{\frac{1}{q}}\right\}   \notag
\end{eqnarray}
\end{remark}

\end{document}